\documentclass[11pt]{amsart}

\usepackage{amssymb}
\usepackage{bm}
\usepackage[centertags]{amsmath}
\usepackage{amsfonts}
\usepackage{amsthm}
\usepackage{mathrsfs}
\usepackage{xcolor}
\usepackage{graphicx}
\usepackage{comment}

\theoremstyle{definition}
\newtheorem{thm}{Theorem}[section]
\newtheorem{dfn}[thm]{Definition}
\newtheorem{lem}[thm]{Lemma}
\newtheorem{prp}[thm]{Proposition}
\newtheorem{cor}[thm]{Corollary}
\newtheorem{rmk}[thm]{Remark}
\newtheorem{rmks}[thm]{Remarks}

\newtheorem{thmA}{Theorem}

\newtheorem{prpA}[thmA]{Proposition}

\newtheorem*{thm*}{Theorem}

\newtheorem*{cor*}{Corollary}

\newtheorem*{prp*}{Proposition}

\newtheorem*{rmk*}{Remark}

\newtheorem*{ntt}{Notation}

\newtheorem*{prb}{Problem}

\newcommand{\N}{\mathbb{N}}
\newcommand{\R}{\mathbb{R}}
\newcommand{\inn}{\in\mathbb{N}}

\newcommand{\e}{\varepsilon}
\newcommand{\al}{\alpha}
\newcommand{\de}{\delta}
\newcommand{\la}{\lambda}

\newcommand{\jc}[1]{_{\tilde{J}(#1)}}

\newcommand{\qpm}{_{q,2}^m}
\newcommand{\jqpm}{_{J_{q,2}^{m}}}
\newcommand{\qpmk}{_{q,2}^{m_k}}
\newcommand{\qpmkprime}{_{q,2}^{m_{k^\prime}}}
\newcommand{\jqpmk}{_{J_{q,2}^{m_{k}}}}
\newcommand{\jqpmkprime}{_{J_{q,2}^{m_{k^\prime}}}}

\DeclareMathOperator{\supp}{supp}

\DeclareMathOperator{\ran}{ran}

\DeclareMathOperator{\dist}{dist}

\long\def\symbolfootnote[#1]#2{\begingroup%
\def\thefootnote{\fnsymbol{footnote}}\footnote[#1]{#2}\endgroup}

\allowdisplaybreaks
\begin{document}

\title[A study of conditional spreading sequences]{A study of conditional spreading sequences}

\author[S. A. Argyros]{Spiros A. Argyros}
\address{National Technical University of Athens, Faculty of Applied Sciences, Department of Mathematics, Zografou Campus, 15780, Athens, Greece}
\email{sargyros@math.ntua.gr}
\author[P. Motakis]{Pavlos Motakis}
\address{Department of Mathematics, Texas A\&M University, College Station, TX 77843-3368, USA}
\email{pavlos@math.tamu.edu}
\author[B. Sar\i]{B\"unyamin Sar\i}
\address{Department of Mathematics, University of North Texas, Denton, TX 76203-5017, USA}
\email{bunyamin@unt.edu}

\thanks{Research of the second author was supported by NSF DMS-1600600.}
\thanks{Research of the third author was supported by grant 208290 from the Simons Foundation.}




\begin{abstract}
It is shown that every conditional spreading sequence can be decomposed into two well behaved parts, one being unconditional and the other being convex block homogeneous, i.e. equivalent to its convex block sequences. This decomposition is then used to prove several results concerning the structure of spaces with conditional spreading bases as well as results in the theory of conditional spreading models. Among other things, it is shown that the space $C(\omega^\omega)$ is universal for all spreading models, i.e., it admits all spreading sequences, both conditional and unconditional, as spreading models. Moreover, every conditional spreading sequence is generated as a spreading model by a sequence in a space that is quasi-reflexive of order one.
\end{abstract}

\maketitle

\setcounter{tocdepth}{1}
\tableofcontents

\section{Introduction}
The notion of a spreading model has been in the heart of the theory of Banach spaces since its conception in 1974 by L. Brunel and A. Sucheston \cite{BS-sm}. A bounded sequence $(x_i)_i$ in a Banach space is said to generate a sequence $(e_i)_i$ in a semi-normed space as a spreading models if for any $n\in\N$ and $\e>0$ there is $n_0\in\N$ with the property that for any $n_0\leqslant k_1<\cdots<k_n$ and scalars $a_1,\ldots,a_n$
\begin{equation}
\left|\left\|\sum_{i=1}^na_ix_{k_i}\right\| - \left\|\sum_{i=1}^na_ie_i\right\|\right| < \e. 
\end{equation}
A Banach space, or a subset of that space, is said to admit $(e_i)_i$ as a spreading model if there exists a sequence $(x_i)_i$ in that set generating $(e_i)_i$ as a spreading model. 
As it was proved in \cite{BS-sm} every bounded sequence in a Banach space has a subsequence generating some spreading model. The first and foremost property of a spreading model is that it is a 1-spreading sequence, i.e. it is isometrically equivalent to its subsequences. A spreading sequence need not be Schauder basic, however this paper is focused on those spreading sequences that are conditional Schauder basic. Typical examples of such sequences are the summing basis of $c_0$ and the boundedly compete basis of James space $J$ from \cite{J}. For a thorough study of basic properties of spreading models and spreading sequences we refer the reader to \cite{BL}.

There have been several applications of the concept of spreading models, a concept that describes the asymptotic behavior of a sequence in a Banach space. It has been utilized as a tool to prove several important results as it can be used to witness canonical infinite dimensional structure exhibited on finite, yet increasingly large, segments of an infinite sequence.
A typical example concerns unconditional structure. Although there exist Banach spaces not containing infinite unconditional sequences (see e.g. \cite{GM}), it is well known and not difficult to show that every infinite dimensional Banach space contains a sequence generating an unconditional spreading model.
Perhaps surprisingly, there is also a strong relationship between the spreading models admitted by a Banach space and the behavior of strictly singular operators on that space. In \cite{AOST} it is shown that if a Banach space $X$ contains sequences that generate spreading models with certain properties, then there exists a subspace $W$ of $X$ that fails the scalar-plus-compact property. In \cite{ADST} the authors expand on this connection by studying the interaction between classes of $\mathcal{S}_\xi$-spreading models and compositions of strictly singular operators. Broadly speaking, understanding the $\mathcal{S}_\xi$-spreading models of a Banach space may provide information about whether compositions of certain strictly singular operators are compact. The strong connection between the theory of spreading models and operator theory is made very clear in \cite{AMisp} where the first two authors constructed the first known example of a reflexive Banach space with the invariant subspace property. The proof of this fact is based on manipulating properties of spreading models in the space that allow to draw conclusions concerning as to when compositions of appropriate operators are compact.

A spreading sequence that is also unconditional is called subsymmetric. As mentioned earlier, the main focus of this paper is to study conditional spreading sequences. Of particular interest is the boundedly complete basis $(e_i)_i$ of James space $J$, the first known Banach space that is quasi-reflexive of order one. In this space the norm of a vector $x = \sum_{i=1}^\infty a_ie_i$ is given by the formula
\begin{equation}
\label{jamesnorm}
\left\|x\right\| = \sup\left\{\left(\sum_{k=1}^n\left(\sum_{i\in E_k}a_i\right)^2\right)^{1/2}\right\},
\end{equation}
where the supremum is taken over all possible choices of successive intervals $(E_k)_{k=1}^n$ of $\N$. Similarly, one may replace the $\ell_2$-sum in \eqref{jamesnorm} with the norm over a subsymmetric sequence $(x_i)_i$ to obtain a conditional spreading sequence, called the jamesification of $(x_i)_i$. This was first considered in \cite{BHO}. A feature shared by all aforementioned spreading sequences is that they are equivalent to their convex block sequences, a property we shall refer to as convex block homogeneity. Sequences with this property are either conditional spreading or they are equivalent to the unit vector basis of $\ell_1$. Interestingly, it also turns out to be related to an isometric definition introduced by Brunel and Sucheston in \cite{BS}, namely that of equal signs additive sequences (see Definition \ref{def esa}, Section \ref{convex block homogeneous bases}). The aforementioned paper seems to be the first attempt aimed towards studying conditional spreading sequences, or at least a subclass of them.
\begin{thmA}
Let $X$ be a Banach space with a Schauder basis $(e_i)_i$. Then, the sequence $(e_i)_i$ is convex block homogeneous if and only if $X$ admits an equivalent norm with respect to which $(x_i)_i$ is equal signs additive. 
\end{thmA}
The most elementary process that leads to conditional spreading bases that are not convex block homogeneous is to take the maximum of a subsymmetric norm and the jamesification of another subsymmetric norm. The natural question to ask is whether this is the unique way by which conditional spreading sequences are obtained. As it is explained in the last section of this paper the answer to this question is negative. Nevertheless, there exists a very nice and useful characterization in terms of decomposing a conditional spreading norm into its constituent parts. In the result below, we naturally identify every Schauder basic sequence with the unit vector basis of $c_{00}(\N)$.
\begin{thmA}
\label{maintheoremitro}
Let $\|\cdot\|$ be a norm on $c_{00}(\N)$ with respect to which the unit vector basis $(e_i)_i$ is conditional spreading. Then, there exist two norms $\|\cdot\|_{u}$ and  $\|\cdot\|_{c}$, both defined on $c_{00}(\N)$ satisfying the following.
\begin{itemize}
\item[(i)] The unit vector basis $(e_i)_i$ is subsymmetric with respect to  $\|\cdot\|_u$.
\item[(ii)] The unit vector basis $(e_i)_i$ is conditional and convex block homogeneous with respect to $\|\cdot\|_c$.
\item[(iii)] There exist positive constants $\kappa$ and $K$ so that for every $x\in c_{00}(\N)$
\begin{equation*}
\kappa\max\left\{\|x\|_{u}, \|x\|_{c}\right\} \leqslant \|x\| \leqslant K\max\left\{\|x\|_{u}, \|x\|_{c}\right\}.
\end{equation*}
\end{itemize}
\end{thmA}
The above theorem can be translated into saying that if $X$ is a Banach space with a conditional spreading basis $(e_i)_i$ then there exist Banach spaces $U$ and $Z$ having Schauder bases $(u_i)_i$ and $(z_i)_i$ that are subsymmetric and convex block homogenous respectively so that $(e_i)_i$ is equivalent to the sequence $(u_i,v_i)_i$ in $U\oplus V$. In other words, $X$ is isomorphic to the diagonal of $U\oplus Z$. We refer to $(u_i)_i$ and $(z_i)_i$ as the unconditional part and the convex block homogeneous part of $(e_i)_i$ respectively. The unconditional part is in fact the sequence $(e_{2i} - e_{2i-1})_i$ whereas the convex block homogeneous part is given by a block sequence of averages of the basis that increase sufficiently rapidly. It was already proved in \cite{FOSZ} that the space $U$ is complemented in $X$, in fact $X\simeq U\oplus X$. It is also true that $Z$ is complemented in $X$ however in interesting cases (i.e. whenever $Z$ is not isomorphic to $c_0$) $X\not\simeq Z\oplus X$.

The aforementioned analysis can be used to prove regularity results for spaces with conditional spreading bases that are uncommon for such broad classes of Banach spaces. This concerns the behavior of block sequences, non-trivial weak Cauchy sequences, as well as complemented subspaces of a space with a conditional spreading basis.
\begin{prpA}
Let $X$ be a Banach space with a conditional spreading basis $(e_i)_i$ and let $(z_i)_i$ be its convex block homogeneous part. Every seminormalized block sequence $(x_i)_i$ in $X$ either has an unconditional subsequence or it has a convex block sequence that is equivalent to $(z_i)_i$. 
\end{prpA}
In fact, every convex block sequence of the basis that is equivalent to $(z_i)_i$ defines a bounded linear projection onto its linear span. We actually prove a somewhat more general result that concerns arbitrary sequences that are not necessarily block.
\begin{prpA}
Let $X$ be a Banach space with a conditional spreading basis $(e_i)_i$ and let $(z_i)_i$ be its convex block homogeneous part. Every sequence in $X$ that is equivalent to $(z_i)_i$ has a subsequence that spans a complemented subspace of $X$.
\end{prpA}
We can combine the above Propositions to obtain a result that concerns complemented subspaces of non-reflexive subspaces of a space $X$ with a conditional spreading basis, namely every non-reflexive subspace of $X$ contains a further subspace that is complemented in $X$. This complemented subspace can be chosen to be very specific, a fact again manifesting the regularity of the space $X$. 
\begin{thmA}
Let $X$ be a Banach space with a conditional spreading basis $(e_i)_i$ and let $(z_i)_i$ be its convex block homogeneous part. Let $(x_i)_i$ be a non-trivial weak Cauchy sequence in $X$. Then, $(x_i)_i$ has a convex block sequence $(w_i)_i$ that is either equivalent to the summing basis of $c_0$ or to $(z_i)_i$. Furthermore, the closed linear span of $(w_i)_i$ is complemented in $X$. 
\end{thmA}
The above result demonstrates the highly homogeneous structure of non-reflexive subspaces of spaces with conditional spreading bases. If the convex block homogeneous part of the basis is not equivalent to the summing basis of $c_0$ one can also deduce the following.
\begin{thmA}
Let $X$ be a Banach space with a conditional spreading basis $(e_i)_i$ and let $(z_i)_i$ be its convex block homogeneous part. Assume that $(z_i)_i$ is not equivalent to the summing basis of $c_0$. Then, if $X = Y\oplus W$, exactly one of the spaces $Y$ and $W$ contain a subspace $\tilde Z$ that is isomorphic to $Z = [(z_i)_i]$. Furthermore, $\tilde Z$ is complemented in $X$.
\end{thmA}
This brings spaces that have a convex block homogeneous basis very close to being primary, a fact that was proved for James space by P. G. Casazza in \cite{C}. Whether such spaces are actually primary or not is unknown to us and it poses an interesting question as to whether such canonical behavior can be witnessed by such a general class of Banach spaces. It is worth mentioning that Rosenthal's dichotomy from \cite{R} plays a very important role in conditional spreading bases. According to this dichotomy, a weakly Cauchy sequence in a Banach space either has a strongly summing subsequence or a convex block sequence that is equivalent to the summing basis of $c_0$. It easily follows that a conditional spreading sequence is strongly summing if and only if its convex block homogeneous part is not equivalent to the summing basis of $c_0$. As it is made evident from the above theorem, strongly summing conditional spreading sequences have special properties. For example, if $X$ has such a basis then $X$ is not isomorphic to its square.

Theorem \ref{maintheoremitro} is not only useful for studying the structure of Banach spaces with conditional spreading bases themselves, it can also be applied to study conditional spreading models of Banach spaces. In fact, the motivation behind the proof of Theorem \ref{maintheoremitro} originates in questions surrounding conditional spreading models. A fairly simple observation is the following.
\begin{prpA}
Every non-reflexive Banach space admits a spreading model that is 1-convex block homogeneous.
\end{prpA}
This ought to be compared to \cite[Theorem 5, page 296]{BS} where it is shown that non-superreflexive Banach spaces that are B-convex have equal signs additive sequences finitely representable in them. In a more specific setting concerning the underlying space, in \cite{O} it is proved that every subsymmetric sequence is admitted as a spreading model by the space $C(\omega^\omega)$. This is interesting because $\omega^\omega$ is the first infinite ordinal number $\alpha$ for which $C(\alpha)$ is not isomorphic to $c_0$, a space with much poorer spreading model structure. In fact, every spreading model of $c_0$ is equivalent to a sequence in $c_0$, namely either unit vector basis of $c_0$ or the summing basis of $c_0$. We extend the result from \cite{O} as described below.
\begin{thmA}
\label{cottoaaps}
The space $C(\omega^\omega)$ admits all possible spreading models, both the conditional and the unconditional ones.
\end{thmA}
The proof involves a construction similar to that of Schreier's space $S$ from \cite{S}, a space that has a weakly null basis generating an $\ell_1$ spreading model, yet the space $S$ embeds into $C(\omega^\omega)$. A noteworthy fact that was proved in \cite{AKT} is that $c_0$ admits every possible spreading sequence as a 2-spreading model. This is a notion of spreading models introduced and studied in \cite{AKT} and \cite{AKThigher} where there is developed an entire theory surrounding the so called $\xi$-spreading models. This theory examines the idea of higher order spreading models in a direction that allows this notion to be taken with respect transfinite ordinal numbers. 

As demonstrated by the results being discussed above, in general, the infinite dimensional structure of the  spreading models generated by a space $X$ can be quite different than the structure of the space $X$ itself. In \cite{B} a reflexive Banach space is constructed that admits the unit vector basis of $\ell_1$ as a spreading model. This is the first known example of a reflexive Banach space admitting a non-reflexive spreading model. It is well known that whenever a Banach space admits a conditional spreading sequence as a spreading model then that space must be non-reflexive. In the spirit of the aforementioned results we show that every conditional spreading sequence is the spreading model of some quasi-reflexive Banach space, i.e. a non-reflexive Banach space that is ``as close as possible to being reflexive''. We furthermore show that every subsymmetric sequence is generated as a spreading model by an unconditional basis of some reflexive Banach space. These two results reveal that a Schauder basic spreading sequence is admitted as a spreading model of a sequence spanning in a sense the ``smallest'' possible type of Banach space,  depending on whether it is conditional or unconditional.
\begin{thmA}
\label{qroocssm}
Let $(x_i)_i$ be a spreading Schauder basic sequence.
\begin{itemize}

  \item[(i)] If $(x_i)_i$ is unconditional, then there exists a reflexive Banach space $X$ with an unconditional Schauder basis $(e_i)_i$ that generates a spreading model equivalent to $(x_i)_i$.

 \item[(ii)] If $(x_i)_i$ is conditional, then there exists a Banach space $X$ that is quasi-reflexive of order one with a Schauder basis $(e_i)_i$ that generates a spreading model equivalent to $(x_i)_i$.
 
\end{itemize}

\end{thmA}

The solution to this problem involves taking a suitable sequence in $C(\omega^\omega)$ and then applying a saturation method with constraints to the norm on that sequence. This is performed in such a way that the desired spreading model is preserved. Related to Theorems \ref{qroocssm} and \ref{cottoaaps} is a result from \cite{AMusm} where a Banach space is constructed admitting all unconditional spreading models in all of its infinite dimensional subspaces. We point out that the space from \cite{AMusm} is hereditarily indecomposable, thus, it does not contain unconditional sequences. This invites the question as to whether there exists a Banach space all subspaces of which admit all possible spreading models, both the conditional and the unconditional ones.

The paper is organized into \ref{counterexamplesection} sections and it can be broken up into four main parts. The first part consists of Sections \ref{preliminaries}, \ref{convex block homogeneous bases}, and \ref{decomposition} in which preliminary facts and definitions are recalled, the definition of convex block homogeneous sequences is studied, and the decomposition of conditional spreading norms is proved. These sections form the foundation for the rest of the paper and all other parts are based on it. In fact, parts two, three, and four can be read independently from one another. The second part consists of Sections \ref{block sequences}, \ref{strongly summing}, \ref{non-trivial weak Cauchy}, \ref{complemented subspaces}, and \ref{Baire-1}. These sections deal with studying the structure of Banach spaces with conditional spreading bases. The third part consists of Sections \ref{nonreflexiveadmitcbhspreadingmodels}, \ref{comegaomegaadmitsall}, and \ref{quasireflexivesm}. In these sections we study conditional spreading sequences as spreading models of certain spaces. The last part consists only of Section \ref{counterexamplesection}. It is devoted to showing that a certain convex block homogeneous sequence is not generated via the jamesification process.

\section{Preliminaries}\label{preliminaries}

We remind some preliminary notions that are integral to the rest of the paper and we state certain known results that will be used repeatedly. We also introduce the concept of the conditional jamesification of a Schauder basic sequence. This is a slight modification of the notion of jamesification introduced in \cite{BHO}.

A sequence $(e_i)_i$ in a Banach space is {\em spreading} if it is equivalent to all of its subsequences and it is called {\em 1-spreading} if the equivalent constants are one. Not all spreading sequences are Schauder basic, however these are the only ones that we shall consider.  The sequence $(e_i)_i$ is {\em subsymmetric} if it is unconditional and spreading. {\em 1-subsymmetric} means 1-unconditional, that is, the norm of vectors  is invariant under changing the signs of coefficients, and 1-spreading. Perhaps, a more often used notion in the theory of spreading sequences is that of {\em suppression unconditionality}; a property of the basis that the norm of vectors does not increase when any subset of coefficients are deleted. In this paper we are mostly interested in {\em conditional spreading} sequences, i.e. Schauder basic spreading sequences that are not unconditional. We say $\|\cdot\|$ is a {\em spreading norm} on $c_{00}(\N)$ if the unit vector basis $(e_i)_i$ is spreading with respect to the norm $\|\cdot\|$.

Let $X$ be a Banach space with a conditional and spreading basis $(e_i)_i$. Then the following are satisfied.
\begin{itemize}

\item[(i)] The summing functional $s(\sum_ia_ie_i) = \sum_ia_i$ is well defined and bounded. If $I$ is an interval of the natural numbers and $P_I$ is the projection onto $I$ associated to the basis $(e_i)_i$, we shall denote
the functional $s\circ P_I$ by $s_I$\label{sIfunctional}.

\item[(ii)] Let $d_1 = e_1$ and $d_{i+1} = e_{i+1} - e_i$ for all $i\inn$. The sequence $(d_i)_i$ forms a Schauder basis for $X$ (\cite[Theorem 2.3 b)]{FOSZ}) and we shall refer to it as {\em the difference basis} of $X$.

\item[(iii)] The sequence $(e_i)_i$ is non-trivial weak Cauchy, that is, weak Cauchy and not weakly convergent. 

\end{itemize}

The following two results were proved in \cite[Theorem 2.3 a) and Theorem 2.8]{FOSZ}, we include their statements as we refer to them quite often in the sequel.

\begin{prp}
\label{summingzero}
Let $X$ be a Banach space with a conditional spreading basis $(e_i)_i$. If $(x_i)_i$ is a block sequence of the basis with $s(x_i) = 0$ for all $i\inn$, then $(x_i)_i$ is unconditional.
If moreover $(e_i)_i$ is 1-spreading, then $(x_i)_i$ is suppression unconditional.
\end{prp}


If a sequence of successive finite intervals $(I_i)_i$ of $\N$ satisfies $\max I_i + 1 = \min I_{i+1}$ for all $i$, then we say that the $I_i$'s are consecutive.

\begin{prp}
\label{averagingprojection}
Let $X$ be a Banach space with a conditional spreading basis $(e_i)_i$ and let also $\bar{I} = (I_i)_i$ be a sequence of consecutive intervals of $\N$.
The map $P_{\bar{I}}:X\rightarrow X$ with
$$P_{\bar{I}} x = \sum_{i=1}^\infty s_{I_i}(x)\left(\frac{1}{\#I_i}\sum_{j\in I_i}e_j\right)$$
is a bounded linear projection. If, moreover, $(e_i)_i$ is 1-spreading and $\cup_iI_i = \N$, then $\|P_{\bar{I}}\| \leqslant 3$.
\end{prp}

\subsection*{Conditional Jamesification}
One way to obtain conditional norms is the following. Let $X$ be a Banach space with a Schauder basis $(x_i)_i$. We define a norm on $c_{00}(\N)$, called the {\em conditional jamesification} of $X$, as follows. For $x\in c_{00}(\N)$ set
\begin{equation}
\label{conditionaljamesificationformula}
\|x\|\jc{X} = \max\left\{\left\|\sum_{i=1}^ns_{I_i}(x)x_i\right\|\right\}
\end{equation}
where the maximum is taken over all consecutive intervals $(I_i)_{i=1}^n$ of $\mathbb{N}$ (i.e. intervals that are successive and have no gaps between them). We also refer to the unit vector basis of $c_{00}(\N)$ endowed with $\|\cdot\|_{\tilde J(X)}$ as the conditional jamesification of $(x_i)_i$. It is clear that if $(x_i)_i$ is spreading, then the unit vector basis $(e_i)_i$ of $c_{00}(\N)$ is spreading with respect to $\|\cdot\|\jc{X}$.

\begin{rmk}
Note that $\|\cdot\|\jc{X}$ really depends on the basis $(x_i)$ of $X$ and $\|\cdot\|\jc{(x_i)}$ would perhaps be a more appropriate notation. However, in what follows the basis in question will always be explicit and will not lead to confusion.
\end{rmk}

\begin{rmk}
We mention that a norm denoted by $\|\cdot\|_{J(X)}$ very similar to the above was considered in \cite{BHO}. The important difference is that in the definition of $\|\cdot\|_{J(X)}$ the intervals are allowed to have gaps whereas in $\|\cdot\|_{\tilde J(X)}$ they are not. We shall adopt the terminology used in that paper and refer to $\|\cdot\|_{J(X)}$ and $\|\cdot\|_{\tilde J(X)}$ as {\em jamesification} and  {\em conditional jamesificiation}, respectively.
\end{rmk}

\begin{rmk}
\label{conditionaljamesificationonsubsymmetric}
The conditional jamesification of a subsymmetric sequence $(x_i)_i$ is equivalent to the jamesification of $(x_i)_i$. In fact, if $(x_i)_i$ is suppression unconditional then its jamesification and its conditional jamesification are isometrically equivalent.
\end{rmk}

\begin{rmk}
\label{conditional jamesification stable}
Iterating the procedure of conditional jamesification will not lead to new norms. That is, if $X$ is a Banach space with a Schauder basis $(x_i)_i$, then the Schauder basis $(e_i)_i$ of the conditional jamesification $\tilde J(X)$ of $X$ is isometrically equivalent to its own conditional jamesification, i.e. $\tilde J(X) = \tilde J(\tilde J(X))$.
\end{rmk}

An easy yet important observation is that conditional spreading sequences dominate their convex block bases.

\begin{lem}
\label{convexdomination}\label{convexdominated}
Let $X$ be a Banach space with a 1-spreading basis $(e_i)_i$. Then, for every $c_1,\ldots,c_n\in\mathbb{R}$ and  convex block vectors $x_1,\ldots,x_n$ of the basis $(e_i)_i$, we have $$\left\|\sum_{i=1}^nc_ix_i\right\| \leqslant \left\|\sum_{i=1}^nc_ie_i\right\|.$$
\end{lem}

\begin{proof}
Assume that $x_k = \sum_{i\in F_k}\la_i^ke_i$ for $k=1,\ldots,n$ and let $x^*$ be a functional with $\|x^*\| = 1$. Choose $i_k\in F_k$ so that $c_kx^*(e_{i_k}) = \max\{c_kx^*(e_{i}): i\in F_k\}$. An easy computation yields the following:
\begin{equation*}
x^*\left(\sum_{k=1}^nc_kx_k\right) \leqslant x^*\left(\sum_{k=1}^nc_ke_{i_k}\right) \leqslant \left\|\sum_{k=1}^nc_ke_{i_k}\right\| = \left\|\sum_{k=1}^nc_ke_k\right\|.
\end{equation*}
\end{proof}


\section{Convex block homogeneous bases}\label{convex block homogeneous bases}
In this section we discuss the central concept introduced in this paper that is also the main tool used herein to study conditional spreading sequences. This is the notion of a convex block homogeneous sequence. Interestingly, it turns out that this is an isomorphic formulation of an isometric property, that of an equal signs additive (ESA) sequence, introduced by A. Brunel and L. Sucheston in \cite{BS}.

\begin{dfn}
\label{cbhdef}
Let $X$ be a Banach space and $(x_i)_i$ be a Schauder basic sequence in $X$. 
\begin{itemize}

\item[(i)] If $(x_i)_i$ is equivalent to all of its convex block sequences then we say that it is convex block homogeneous.

\item[(ii)] If $(x_i)_i$ is isometrically equivalent to all of its convex block sequences then we say that it is 1-convex block homogeneous.

\end{itemize}
\end{dfn}
A convex block homogeneous sequence is clearly spreading, and if it is unconditional, then it is equivalent to the unit vector basis of $\ell_1$.  A 1-convex block homogeneous sequence is clearly 1-spreading. Furthermore, by a  standard argument, if a basis is convex block homogeneous then there exists a constant $C$ so that it is $C$-equivalent to all of its convex block sequences.

The simplest examples of convex block homogeneous bases are the unit vector basis of $\ell_1$ and the summing basis of $c_0$. Another classical example is the boundedly complete basis of James space \cite{J}. In fact, the jamesification of any subsymmetric sequence yields a convex block homogeneous basis. In a later section we investigate whether these are the only possible convex block homogeneous bases.

The first result of this section is a  characterization of convex block homogeneous bases.

\begin{prp}\label{equivalentdefinitions}
Let $X$ be a Banach space with a conditional spreading basis $(e_i)_i$. The following statements are equivalent.
\begin{itemize}

\item[(i)] The basis $(e_i)_i$ is equivalent to its conditional jamesification (see \eqref{conditionaljamesificationformula}, page \pageref{conditionaljamesificationformula}). 

\item[(ii)] The sequence $(s_{[1,n]})_n$ is spreading, i.e. equivalent to its subsequences (for the definition of $(s_{[1,n]})_n$ see page \pageref{sIfunctional}).

\item[(iii)] The basis $(e_i)_i$ is convex block homogeneous.

\item[(iv)] Every block sequence $(x_n)_n$ of averages of the basis is equivalent to $(e_i)_i$.

\end{itemize}
\end{prp}

\begin{proof}
We may assume without loss of generality that $(e_i)_i$ is 1-spreading.

(i)$\Rightarrow$(ii): Choose $C>0$ such that for every $x\in c_{00}$ and consecutive intervals $(I_i)_{i=1}^n$ of $\mathbb{N}$ we have $\|\sum_{i=1}^ns_i(x)e_i\| \leqslant C\|x\|$. Let now $n\inn$, $\la_1,\ldots,\la_n\in\mathbb{R}$ and $k_1<\cdots<k_n\in\mathbb{N}$. An easy argument using the spreading property of the basis $(e_i)_i$ yields that $\|\sum_{i=1}^n\la_is_{[1,i]}\| \leqslant \|\sum_{i=1}^n\la_is_{[1,k_i]}\|$. Let now $x\in c_{00}(\N)$ with $\|x\| =1$, set $I_1 = \{1,\ldots,k_1\}, I_i = \{k_{i-1}+1,\ldots,k_i\}$ for $i\geqslant 2$ and set $\tilde{x} = \sum_{i=1}^ns_{I_i}(x)e_i$. We have that $\|\tilde{x}\| \leqslant C$. A calculation yields the following:
\begin{equation*}
\sum_{i=1}^n\la_is_{[1,k_i]}(x) = \sum_{i=1}^n\la_is_{[1,i]}(\tilde{x}) \leqslant C\left\|\sum_{i=1}^n\la_is_{[1,i]}\right\|.
\end{equation*}

(ii)$\Rightarrow$(iii): Choose $C>0$ so that for every $n\inn$, $\la_1,\ldots,\la_n\in\mathbb{R}$ and $k_1<\cdots<k_n\in\mathbb{N}$ we have $\|\sum_{i=1}^n\la_is_{[1,k_i]}\| \leqslant C\|\sum_{i=1}^n\la_is_{[1,i]}\|$. Let $(x_k)_k$ be a convex block sequence of the basis, where for each $k$ we have $x_k = \sum_{i\in F_k}\la_i^ke_i$ and let $c_1,\ldots,c_n\in\mathbb{R}$. By Lemma \ref{convexdomination} we immediately obtain that $\|\sum_{i=1}^n\la_ix_i\| \leqslant \|\sum_{i=1}^n\la_ie_i\|$.
On the other hand, let $x^* = \sum_{k=1}^m\mu_ks_{[1,k]}$, with $\|x^*\| = 1$. Set $j_k = \max F_k$. If $y^* = \sum_{k=1}^m\mu_ks_{[1,j_k]}$, then $\|y^*\|\leqslant C$. A calculation yields the following.
\begin{equation*}
x^*\left(\sum_{k=1}^nc_ke_k\right) = \sum_{k=1}^nc_kx^*(e_k) = \sum_{k=1}^nc_ky^*(x_k) \leqslant C\left\|\sum_{k=1}^nc_kx_k\right\|.
\end{equation*}

(iii)$\Rightarrow$(iv): This is trivial.

(iv)$\Rightarrow$(i): For simplicity, we assume that $(e_i)_i$ is bimonotone. Choose $C>0$ such that for every $\la_1,\ldots,\la_n\in\mathbb{R}$ and  block vectors $x_1,\ldots,x_n$ that are averages of the basis, we have that $\|\sum_{i=1}^n\la_ie_i\| \leqslant C\|\sum_{i=1}^n\la_ix_i\|$. Let $x\in c_{00}(\N)$ and $I_1<\cdots<I_n$ be consecutive intervals of $\mathbb{N}$. Set $x_i = (\#I_i)^{-1}\sum_{j\in I_i}e_j$ and let $\tilde{x}$ the restriction of $x$ onto the interval $\cup_is_i$. Then $\|\tilde{x}\| \leqslant \|x\|$. Using Proposition \ref{averagingprojection} we conclude:
\begin{eqnarray*}
\left\|\sum_{i=1}^ns_{I_i}(x)e_i\right\| &\leqslant& C\left\|\sum_{i=1}^ns_{I_i}(x)x_i\right\| = C\left\|\sum_{i=1}^ns_{I_i}(\tilde{x})\left(\frac{1}{\#I_i}\sum_{j\in I_i}e_j\right)\right\|\\
 &\leqslant& 3C\|\tilde{x}\| \leqslant 3C\|x\|.
\end{eqnarray*}
\end{proof}

\subsection*{Equal signs additive (ESA) bases}

The following isometric definitions are from \cite[p.~288]{BS}. As it is proved there, \cite[Lemma 1]{BS}, they are equivalent.
\begin{dfn}\label{def esa}
Let $X$ be a Banach space with a Schauder basis $(e_i)_i$.
\begin{itemize}

\item[(i)] The basis $(e_i)_i$ is called equal signs additive, if for any sequence of scalars $(a_k)_k$, for any natural number $k$ so that $a_ka_{k+1}\geqslant 0$:
$$\left\|\sum_{i=1}^{k-1}a_ie_i + (a_k + a_{k+1})e_k + \sum_{i=k+2}^\infty a_ie_i\right\| = \left\|\sum_{i=1}^\infty a_ie_i\right\|.$$

\item[(ii)] The basis $(e_i)_i$ is called subadditive, if for any sequence of scalars $(a_k)_k$, for any natural number $k$:
$$\left\|\sum_{i=1}^{k-1}a_ie_i + (a_k + a_{k+1})e_k + \sum_{i=k+2}^\infty a_ie_i\right\| \leqslant \left\|\sum_{i=1}^\infty a_ie_i\right\|.$$

\end{itemize}
\end{dfn}

As it turns out, convex block homogeneity characterizes sequences that admit equivalent equal signs additive norms. Actually, a sequence is equal signs additive if and only if it is 1-convex block homogeneous.

\begin{prp}
\label{isometriccharacterization}
Let $X$ be a Banach space with a Schauder basis $(e_i)_i$. Then, the sequence $(e_i)_i$ is equal signs additive if and only if it is 1-convex block homogeneous. 
\end{prp}

\begin{proof}
Assume that $(e_i)_i$ is 1-convex block homogeneous and let $(a_i)_i$ be a sequence of scalars. Let $k\in\N$ with $a_ka_{k+1}\geqslant 0$. We will show that Definition \ref{def esa} (i) is satisfied. If $a_k+a_{k+1} = 0$ then this is trivial. Otherwise, define the sequence $(x_i)_i$ with $x_i = e_i$ for $i<k$, $x_k = (a_k/(a_k + a_{k+1}))e_k + (a_{k+1}/(a_k+a_{k+1}))e_{k+1})$ and $x_i = e_{i+1}$ for $i>k$. Then, $(x_i)_i$ is a convex block sequence of the basis and hence isometrically equivalent to it. This yields
\begin{align*}
&\left\|\sum_{i=1}^{k-1}a_ie_i + (a_k + a_{k+1})e_k + \sum_{i=k+2}^\infty a_ie_i\right\|\\ &= \left\|\sum_{i=1}^{k-1}a_ix_i + (a_k + a_{k+1})x_k + \sum_{i=k+2}^\infty a_ix_i\right\| = \left\|\sum_{i=1}^{k+1} a_ie_i+\sum_{i=k+2}^\infty a_ie_{i+1}\right\|\\
&=\left\|\sum_{i=1}^\infty a_ie_i\right\|\text{ (by the spreading property of }(e_i)_i).
\end{align*}
Conversely, assume that $(e_i)_i$ is equal signs additive. A finite induction on $m\in\N$ applied to Definition \ref{def esa} (i) yields that for $k\in\N$ so that $a_k,\ldots,a_{k+m}$ are either all non-negative or all non-positive we have
\begin{equation*}
\left\|\sum_{i=1}^{k-1}a_ie_i + (a_k + \cdots + a_{k+m})e_k + \sum_{i=k+m+1}^\infty a_ie_i\right\| = \left\|\sum_{i=1}^\infty a_ie_i\right\|.
\end{equation*}
This implies that if $(E_k)_k$ is a partition of $\N$ into successive intervals, then for any choice of scalars so that for each $k\in\N$ the scalars $a_i$, $i\in E_k$ are either all non-negative or all non-positive we have
\begin{equation}
\label{esalongvectors}
\left\|\sum_{k=1}^\infty\left(\sum_{i\in E_k}a_i\right)e_k\right\| = \left\|\sum_{i=1}^\infty a_ie_i\right\|.
\end{equation}
Now, if $(x_k)_k$ is a convex block sequence of $(e_i)_i$, $(E_k)_k$ is a partition of $\N$ into successive intervals of $\N$, and $(c_i)_{i\in E_k}$, $k\in\N$ are finite sequences of non-negative scalars summing up to one with $x_k = \sum_{i\in E_k}c_ie_i$, then for any choice of scalars $(a_k)_k$ we have
\begin{equation*}
\begin{split}
\left\|\sum_{k=1}^\infty a_ke_k\right\| =& \left\|\sum_{k=1}^\infty \left(a_k\sum_{i\in E_k}c_i\right)e_k\right\| = \left\|\sum_{k=1}^\infty a_k\left(\sum_{i\in E_k}c_ie_i\right)\right\|\text{ (by \eqref{esalongvectors})}\\
=& \left\|\sum_{k=1}^\infty a_kx_k\right\|,
\end{split}
\end{equation*}
i.e. $(x_k)_k$ isometrically equivalent to the basis $(e_i)_i$ and the proof is complete.
\end{proof}

\begin{cor}
\label{renormto1cbh}
Let $X$ be a Banach space with a convex block homogeneous basis $(e_i)_i$. Then $X$ admits an equivalent norm with respect to which $(e_i)_i$ is bimonotone and equal signs additive. In particular, $X$ admits an equivalent norm with which $(e_i)_i$ is bimonotone and 1-convex block homogeneous.
\end{cor}

\begin{proof}
By passing to a first equivalent norm we may assume that $(e_i)_i$ is 1-spreading. We then pass to the conditional jamesification of $X$ which, by Proposition \ref{equivalentdefinitions} is equivalent to the initial norm. It is also clear that this is a bimonotone norm and also that it is 1-spreading. By Remark \ref{conditional jamesification stable} one can easily deduce that the norm is also subadditive and hence, by \cite[Lemma 1]{BS} it is equal signs additive. By Proposition \ref{isometriccharacterization} the sequence $(e_i)_i$ is 1-convex block homogeneous.
\end{proof}

As a consequence we obtain the following equivalence between these notions as mentioned above.

\begin{thm}\label{cbh is esa}
Let $X$ be a Banach space with a Schauder basis $(e_i)_i$. The following are equivalent.
\begin{itemize}

\item[(i)] The basis $(e_i)_i$ is convex block homogeneous.

\item[(ii)] The space $X$ admits an equivalent norm with respect to which $(e_i)_i$ is equal signs additive.

\end{itemize}
\end{thm}

\begin{rmk}
\label{esabimonotoneisthesameasbeingequaltoconditionaljamesification}
Similar arguments as those used above imply that the basis $(e_i)_i$ of $X$ is bimonotone and equal signs additive if and only if it is 1-spreading and equal to its conditional jamesification in the sense of \eqref{conditionaljamesificationformula}.
\end{rmk}

The following result is of no interest from an isomorphic scope, however it removes the annoying factor three in Proposition \ref{averagingprojection} in case one has a 1-convex block homogeneous basis instead of just a 1-spreading Schauder basic sequence.

\begin{prp}
\label{inthisnicecaseprojectionhasnormone}
Let $X$ be a Banach space with a is a 1-convex block homogeneous Schauder basis $(e_i)_i$ and let $\bar I = (I_k)_k$ be a partition of $\N$ into successive intervals of natural numbers. Then, the map $P_{\bar I}:X\to X$ with
$$P_{\bar{I}} x = \sum_{k=1}^\infty s_{I_k}(x)\left(\frac{1}{\#I_k}\sum_{i\in I_k}e_i\right)$$
is a norm-one linear projection.
\end{prp}

\begin{proof}
By Proposition \ref{averagingprojection} $P_{\bar I}$ is indeed bounded and it has norm at most three. By \cite[Lemma 1]{BS}  the basis $(e_i)_i$ satisfies Definition \ref{def esa} (ii). An argument identical to that used in the of Proposition \ref{isometriccharacterization} yields that for any scalars $(a_i)_i$ we have
\begin{equation}
\label{subadditivelongvectors}
\left\|\sum_{k=1}^\infty\left(\sum_{i\in I_k}a_i\right)e_k\right\| \leqslant \left\|\sum_{i=1}^\infty a_ie_i\right\|.
\end{equation}
Let now $x = \sum_{i=1}^\infty a_ie_i$ be a vector in $X$. Then,
\begin{align*}
 \left\|P_{\bar I}x\right\| =& \left\|\sum_{k=1}^\infty s_{I_k}(x)\left(\frac{1}{\#I_k}\sum_{i\in I_k}e_i\right)\right\|\\
 \leqslant & \left\|\sum_{k=1}^\infty s_{I_k}(x)e_k\right\|\text{ (by Lemma \ref{convexdominated})}\\
 =&\left\|\sum_{k=1}^\infty\left(\sum_{i\in I_k}a_i\right)e_k\right\| \leqslant \left\|\sum_{i=1}^\infty a_ie_i\right\| = \|x\| \text{ (by \eqref{subadditivelongvectors})}.
\end{align*}
\end{proof}

We record the following duality result that can be proved directly without the use of equal signs additive bases, however, this tool makes the proof immediate. 
\begin{prp}
\label{dualizecbh}
Let $X$ be a Banach space with a spreading basis $(e_i)_i$. Then, $(e_i)_i$ is convex block homogeneous if and only if $(s_{[1,n]})_n$ is convex block homogeneous.
\end{prp}

\begin{proof}
If $(s_{[1,n]})_n$ is convex block homogeneous then it is spreading and hence by Proposition \ref{equivalentdefinitions} (ii)$\Rightarrow$(iii) $(e_i)_i$ is convex block homogeneous. If, on the other hand, $(e_i)_i$ is convex block homogeneous then by Corollary \ref{renormto1cbh} we may assume that it is equal signs additive. By \cite[Proposition 2, page 291]{BS} $(s_{[1,n]})_{n\geqslant 2}$ is equal signs additive and hence, convex block homogeneous. By Proposition \ref{equivalentdefinitions} (iii)$\Rightarrow$(ii) $(s_{[1,n]})_{n}$ is equivalent to $(s_{[1,n]})_{n\geqslant 2}$.
\end{proof}

\section{A characterization: decomposing conditional spreading norms}\label{decomposition}




We devote this section to the statement and proof of the main result used throughout the paper, namely the decomposition of conditional spreading norms into two parts. The first part is subsymmetric and it is simply defined by taking skipped successive differences of the basis, whereas the second part is convex block homogeneous and it is obtained by taking averages of the basis that are growing sufficiently rapidly.

\begin{thm}\label{spreading-characterization}
Let $X$ be a Banach space with a conditional 1-spreading basis $(e_i)_i$. Let also $(u_i)_i$ denote the subsymmetric skipped difference sequence of $(e_i)_i$, i.e. $u_i = e_{2i} - e_{2i-1}$ for all $i\in\N$. Then, there exists a Banach space $Z$ with a 1-convex block homogeneous basis $(z_i)_i$ so that for all scalars $a_1,\ldots,a_n$ we have
\begin{equation}
\label{Mainmainmainmaintheoremequation}
\begin{split}
 \frac{1}{2}\max\left\{\left\|\sum_{i=1}^na_iu_i\right\|, \left\|\sum_{i=1}^na_iz_i\right\|\right\} \leqslant \left\|\sum_{i=1}^na_ie_i\right\|\\ \leqslant 2\max\left\{\left\|\sum_{i=1}^na_iu_i\right\|, \left\|\sum_{i=1}^na_iz_i\right\|\right\}.
\end{split}
 \end{equation}
Furthermore, the sequence $(u_i)_i $ is suppression unconditional and the norm on $Z$ is given by
\begin{equation}
\label{what is the convex block homogeneous part exactly formula}
\begin{split}
\left\|\sum_{i=1}^na_iz_i\right\| &= \lim_{N\to\infty}\left\|\sum_{i=1}^na_i\left(\frac{1}{N}\sum_{j = (i-1)N + 1}^{iN}e_j\right)\right\|\\
&= \inf\left\{\left\|\sum_{i=1}^na_iw_i\right\|: (w_i)_{i=1}^n\text{ is a conv. block seq. of }(e_i)_i\right\}.
\end{split}
\end{equation}
\end{thm}


\begin{rmk}
\label{equivalentmaxnorm}
It is useful to restate the theorem in an isomorphic setting as follows. Let $X$ be a Banach space with a conditional spreading basis $(e_i)_i$. The sequence $(e_i)_i$ may be assumed to be 1-spreading. Let $(u_i)_i$ and $(z_i)_i$ be as above, and set $U = [(u_i)_i]$ and $Z = [(z_i)_i]$. Then, the theorem asserts that the map
\begin{align*}
T:X\to U\oplus Z\\
Te_i = (u_i,z_i)
\end{align*}
is an isomorphic embedding. In this decomposition we shall refer to $(u_i)_i$ and $(z_i)_i$ as unconditional and convex block homogenous parts of $(e_i)_i$, respectively. 

Moreover, as we shall observe in next section, $(z_i)$ is equivalent to a block basis of $(e_i)$ (Proposition \ref{theoremaboutblocksequences}), and it is unique in the following sense. If $(\tilde z_i)$ is a convex block homogeneous sequence spanning a space $\tilde Z$ so that the map $\tilde T:X\to U\oplus \tilde Z$ defined by $\tilde Te_i = (u_i,\tilde z_i)$ is an isomorphic embedding, then $(\tilde z_i)$ is equivalent to $(z_i)$. 
\end{rmk}

The proof of Theorem \ref{spreading-characterization}, which is postponed until the end of the section, is based on two lemmas given below.

\begin{lem}\label{convexandunconditionaldomination}
Let $X$ be a Banach space with a 1-spreading basis $(e_i)_i$. Then, for every scalars $c_1,\ldots,c_n$ and  convex block  vectors $x_1,\ldots,x_n$ of the basis:
\begin{equation*}
\left\|\sum_{k=1}^nc_ke_k\right\| \leqslant \left\|\sum_{k=1}^nc_kx_k\right\| + \left\|\sum_{k=1}^nc_k(e_{2k-1} - e_{2k})\right\|.
\end{equation*}
\end{lem}

\begin{proof}
Assume that $x_k = \sum_{j\in F_k}\la_j^ke_j$ for $k=1,\ldots,n$. Using the spreading property of the basis, we may clearly assume that there exist natural number $i_1<\cdots<i_n$ with $i_1 < F_1 <i_2< F_2<\cdots<i_n < F_n$. Let $x^*$ be a functional with $\|x^*\| = 1$. Choose $j_k\in F_k$ with $c_kx^*(e_{j_k}) = \min\{c_kx^*(e_{j}): j\in F_k\}$. An easy computation yields the following.
\begin{eqnarray*}
x^*\left(\sum_{k=1}^nc_ke_{i_k}\right) &=& x^*\left(\sum_{k=1}^nc_k(e_{i_k} - e_{j_k})\right) + x^*\left(\sum_{k=1}^nc_ke_{j_k}\right)\\
&\leqslant& x^*\left(\sum_{k=1}^nc_k(e_{i_k} - e_{j_k})\right) + x^*\left(\sum_{k=1}^nc_kx_k\right)\\
&\leqslant& \left\|\sum_{k=1}^nc_k(e_{i_k} - e_{j_k})\right\| + \left\|\sum_{k=1}^nc_kx_k\right\|\\
&=& \left\|\sum_{k=1}^nc_k(e_{2k-1} - e_{2k})\right\| + \left\|\sum_{k=1}^nc_kx_k\right\|.
\end{eqnarray*}
The fact that $(e_i)_i$ is 1-spreading finishes the proof.
\end{proof}

\begin{lem}\label{averagesdecrease}
Let $X$ be a Banach space with a 1-spreading basis $(e_i)_i$ and let $w_1,\ldots,w_n$ be convex block vectors of the basis. Then for every $\e>0$ there exists $M_0\in\mathbb{N}$, such that for every $c_1,\ldots,c_n\in[-1,1]$ and block averages of the basis $y_1,\ldots,y_n$, with $\#\supp y_k \geqslant M_0$ for $k=1,\ldots,n$, the following holds:
\begin{equation*}
\left\|\sum_{k=1}^nc_ky_k\right\| < \left\|\sum_{k=1}^nc_kw_k\right\| + \e.
\end{equation*}
\end{lem}

\begin{proof}
Allowing some error, say $\e/2$, we may assume that the $w_k$'s have rational coefficients, i.e. there are some $N\in\N$, successive sets $H_1,\ldots,H_n$ and $(n_j)_{j\in H_k}$ so that $\sum_{j\in H_k}(n_j/N) = 1$ and $w_k = \sum_{j\in H_k}(n_j/N)e_j$ for $k=1,\ldots,n$. Note that $N = \sum_{j\in H_k}n_j$ for $k=1,\ldots,n$. Choose $M_0 > 4n\e^{-1}N$.

Let now $y_k =(1/m_k)\sum_{i\in G_k}e_i$, with $\# G_k=m_k$, $m_k\geqslant M_0$, for $k=1,\ldots,n$ and $G_1<\dots<G_n$. Set $r_k = \lfloor m_k/N\rfloor$, choose $G_k^\prime\subset G_k$ with $\#G_k^\prime = Nr_k$ and set $y_k^\prime = (1/(Nr_k))\sum_{j\in G_k^\prime}e_j$. Some computations yield that $\|y_k - y_k^\prime\| < \e/(2n)$. It is therefore sufficient to prove that $\|\sum_{k=1}^nc_ky_k^\prime\| \leqslant \|\sum_{k=1}^nc_kw_k\|$.

Partition each $G_k^\prime$ into further successive sets $(E_j^k)_{j\in H_k}$ with $\#E_j^k = r_kn_j$ and define $z_j^k = (1/n_jr_k)\sum_{i\in E_j^k}e_i$ for $k=1,\ldots,n$ and $j\in H_k$.


Applying Lemma \ref{convexdomination} we finally conclude the following:
\begin{equation*}
\left\|\sum_{k=1}^nc_ky_k^\prime\right\| = \left\|\sum_{k=1}^nc_k\left(\sum_{j\in H_k}\frac{r_kn_j}{Nr_k}z_j^k\right)\right\| \leqslant \left\|\sum_{k=1}^nc_k\sum_{j\in H_k}\frac{n_j}{N}e_j\right\| = \left\|\sum_{k=1}^nc_kw_k\right\|.
\end{equation*}
\end{proof}

%
%
%
%
%
%
%

\begin{proof}[Proof of Theorem \ref{spreading-characterization}]
Lemma \ref{averagesdecrease} clearly yields that any pair of block sequences $(x_k)_k, (y_k)_k$ of averages of the basis with $\#\supp(x_k),\#\supp(y_k)$ tending to infinity must admit the same spreading model. We name this spreading model  $(z_i)_i$ and denote its closed linear span $Z$. Lemma \ref{averagesdecrease} also clearly implies \eqref{what is the convex block homogeneous part exactly formula}. The second line of \eqref{what is the convex block homogeneous part exactly formula} and Lemma \ref{convexdomination} easily imply that $(z_i)_i$ is 1-convex block homogeneous.

To finish the proof we need to show \eqref{Mainmainmainmaintheoremequation}. It follows from Lemmas \ref{convexdomination} and \ref{convexandunconditionaldomination} that for all real scalars $(c_k)_{k=1}^n$
\begin{equation*}
\left\|\sum_{k=1}^nc_kz_k\right\| \leqslant \left\|\sum_{k=1}^nc_ke_k\right\| \leqslant \left\|\sum_{k=1}^nc_kz_k\right\| + \left\|\sum_{k=1}^nc_k(e_{2k-1} - e_{2k})\right\|,
\end{equation*}
which implies conclusion.
\end{proof}

\begin{rmk}
It is tempting to conjecture that every convex block homogeneous norm is equivalent to the jamesification of a subsymmetric norm, and consequently every conditional spreading norm is up to equivalence generated by two subsymmetric norms. However, this turned out to be false. We present a counterexample in Section \ref{counterexamplesection}.
\end{rmk}


\section{Block sequences of conditional spreading bases}\label{block sequences}
This section is centered around understanding the structure of block sequences in a space with a conditional spreading basis. Although the basic statements are included in the proposition below, more precise information is given the subsequent lemmas.

\begin{prp}
\label{theoremaboutblocksequences}
Let $X$ be a Banach space with a conditional spreading basis $(e_i)_i$ and let $(u_i)_i$, $(z_i)_i$ be the unconditional and convex block homogeneous parts of $(e_i)_i$ respectively. The following hold.
\begin{itemize}
 \item[(i)] There exists a block sequence $(\tilde z_i)_i$ of averages of the basis $(e_i)_i$ that is equivalent to $(z_i)_i$.
 \item[(ii)] The closed linear span of every convex block sequence $(x_i)_i$ of the basis $(e_i)_i$ that is equivalent to $(z_i)_i$ is complemented in $X$. In particular, $Z = [(z_i)_i]$ is isomorphic to a complemented subspace of $X$.
 \item[(iii)] Every convex block sequence $(x_i)_i$ of the basis has a subsequence that is equivalent to either $(e_i)_i$ or to $(z_i)_i$.
 \item[(iv)] Every convex block sequence $(x_i)_i$ of the basis $(e_i)_i$ has a further convex block sequence that is equivalent to $(z_i)_i$.
 \item[(v)] Every seminormalized block sequence $(x_i)_i$ of the basis $(e_i)_i$ either has a subsequence that is unconditional or it has a convex block sequence that is equivalent to $(z_i)_i$.
\end{itemize}
\end{prp}

The following lemma is well known. Its proof, which we omit, is based based on a counting argument.

\begin{lem}\label{rosenthaldichotomy}
Let $(u_i)_i$ be a subsymmetric sequence in some Banach space, that is not equivalent to the unit vector basis of $\ell_1$. Then for every $\e>0$ there exists $\de>0$ such that for any real numbers $(a_i)_i$ with $\sum_i|a_i| \leqslant 1$ and $\sup_i|a_i| <\de$ we have that $\|\sum_ia_iu_i\| < \e$.
\end{lem}

The following lemma proves Proposition \ref{theoremaboutblocksequences} (i).

\begin{lem}\label{Zxembeds}
Let $X$ be a Banach space with a conditional spreading basis $(e_i)_i$ and let $(z_i)_i$ be its convex block homogeneous part.
Then, there exists a sequence of natural numbers $(n_i)_i$ so that for every sequence  $(A_i)_i$ of successive subsets of $\N$ with $\#A_i \geqslant n_i$ for all $i\inn$ the following holds:
if $\tilde{z}_i = (1/\#A_i)\sum_{j\in A_i}e_j$ for $i\inn$, then the sequence $(\tilde{z}_i)_i$ is equivalent to $(z_i)_i$.
\end{lem}

\begin{proof}
Let $(u_i)_i$ be the unconditional part of $(e_i)_i$ and applying Lemma \ref{rosenthaldichotomy}, for every $i\inn$ choose $n_i\inn$ so that for every convex combination $u$ of $(u_i)_i$ with $\|u\|_\infty \leqslant n_i^{-1}$,
we have that $\|u\| < 2^{-i}$.

Let now $(A_i)_i$ be successive subsets of $\N$ with  $\#A_i \geqslant n_i$ for all $i\inn$ and set $\tilde{z}_i = (1/\#A_i)\sum_{j\in A_i}e_j$ for $i\inn$.
Let $T$ be the isomorphic embedding from Remark \ref{equivalentmaxnorm}. By the choice of the sequence $(n_i)_i$ we have that $\sum_i\|P_{U_X}T\tilde{z}_i\| <1$ which implies that $(\tilde{z}_i)_i$ is equivalent to
$(P_{Z_X}T\tilde{z}_i)_i$ which is a convex block sequence of $(z_i)_i$ and hence equivalent to $(z_i)_i$.
\end{proof}


\begin{rmk}\label{convexblockhomogeneouspartisunique}
The proof of Lemma \ref{Zxembeds} implies that the sequence $(z_i)_i$ is unique in the sense explained in Remark \ref{equivalentmaxnorm}. Indeed, if $(w_i)_i$ is a convex block homogeneous sequence so that the map $\tilde T:X\to U\oplus\tilde W$ with $W = [(w_i)_i]$ and $\tilde Te_i = (u_i,w_i)$ is an isomorphic embedding,  repeating the argument from the proof of Lemma \ref{Zxembeds}, if $(\tilde z_i)_i$ is as in the statement of that lemma, we conclude that $(\tilde z_i)_i$ is equivalent to $(z_i)_i$ as well as to $(w_i)_i$.
\end{rmk}


The following Lemma proves Proposition \ref{theoremaboutblocksequences} (ii).

\begin{lem}\label{Zxconvexblockiscomplemented}
Let $X$ be a Banach space with a conditional spreading basis $(e_i)_i$ and let $(z_i)_i$ be its convex block homogeneous part. Let also $(x_i)_i$ be a convex block sequence of the basis which is equivalent to $(z_i)_i$.
Then for every sequence of consecutive intervals $(I_i)_i$ of $\N$  with $\supp x_i\subset I_i$ for all $i$, we have that the map $P:X\rightarrow X$
with $Px = \sum_{i=1}^\infty s_{I_i}(x)x_i$ is a bounded linear projection.
\end{lem}

\begin{proof}
We may assume that $(e_i)_i$ is 1-spreading. If we define for $i\in\N$ the vector $y_i = (1/\#I_i)\sum_{j\in I_i}e_j$ then Proposition \ref{averagingprojection} yields that the map $P_{\bar{I}}:X\rightarrow X$ with $P_{\bar{I}}x = \sum_{i=1}^\infty s_{I_i}(x)y_i$ is a bounded linear projection. Observe that by \eqref{what is the convex block homogeneous part exactly formula} the map $R:[(y_i)]\to [(z_i)_i]$ with $Ry_i = z_i$ has norm one. We easily conclude that if $S:[(z_i)]\to [(x_i)]$ is the isomorphism given by $Sz_i = x_i$ then $Q = SRP$ is a bounded linear projection.
\end{proof}

\begin{rmk}
As it is implied by Proposition \ref{theoremaboutblocksequences} (ii), if $X$ is  Banach space with a convex block homogeneous basis $(e_i)_i$, then every convex block sequence of the basis spans a complemented subspace of $X$. In particular, the space spanned by every subsequence of the basis is complemented, without the basis being unconditional.  
\end{rmk}


\begin{proof}[Proof of Proposition \ref{theoremaboutblocksequences} (iii) and (iv)]
We first observe that statements (i) and (iii) of Proposition \ref{theoremaboutblocksequences} immediately imply statement (iv), i.e. we only need to prove statement (iii).
Let $(x_i)_i$ be a convex block sequence of $(e_i)_i$ and let $T$ be the isomorphic embedding from Remark \ref{equivalentmaxnorm}. Assume first that $\|x_i\|_\infty\rightarrow 0$.
Lemma \ref{rosenthaldichotomy} implies that $\|P_{U_X}Tx_i\|\rightarrow 0$ and therefore $(x_i)_i$ has a subsequence equivalent to $(P_{Z_X}Tx_i)_i$ which is a convex block sequence of $(z_i)_i$ and hence equivalent to $(z_i)_i$.

Otherwise we may assume that there exists $\e>0$ such that $\|x_i\|_\infty>\e$ for all $i\inn$. In this case if $v_i =  P_{U_X}Tx_i$ for $i\inn$, then $(v_i)_i$ is a convex block sequence of $(u_i)_i$ with $\|v_i\|_\infty>\e $ for all $i\inn$.
Lemma \ref{convexdominated} and unconditionality imply that $(v_i)_i$ is equivalent to $(u_i)_i$. Moreover, if $w_i = P_{Z_X}Tx_i$, then $(w_i)_i$ is a convex block sequence of $(z_i)_i$ and hence equivalent to $(z_i)_i$.
We conclude that $(v_i, w_i)_i$ is equivalent to $(u_i,z_i)_i$ which is equivalent to $(e_i)_i$. Since $Tx_i = (v_i, w_i)$ for all $i\inn$ and $T$ is an isomorphic embedding, we have that $(x_i)_i$ is equivalent to $(e_i)_i$.
\end{proof}




\begin{proof}[Proof of Proposition \ref{theoremaboutblocksequences} (v)]
If $(x_i)_i$ has a subsequence equivalent to the unit vector basis of $\ell_1$, then obviously there is nothing more to prove. We may therefore assume that $(x_i)_i$ is weak Cauchy.

If $\lim_is(x_i) = 0$, then passing to a subsequence and perturbing we may assume that $s(x_i) = 0$ for all $i\inn$. By Proposition \ref{summingzero} $(x_i)_i$ is then unconditional. Otherwise, by passing to some subsequence of $(x_i)_i$, perturbing and scaling, we may assume that $s(x_i) = 1$ for all $i\inn$.
Choose $k_i\in\supp x_i$ for all $i\inn$ and set $y_i = x_i - e_{k_i}$. Observe that $s(y_i)=0$ for all $i\inn$ and therefore by Proposition \ref{summingzero} $(y_i)_i$ is unconditional. Moreover, as both $(x_i)_i$ and $(e_i)_i$ are weak Cauchy,
the same is true for $(y_i)_i$, which implies that $(y_i)_i$ is weakly null. Using Mazur's Theorem we conclude that there exists a convex block sequence of $(x_i)_i$ which is equivalent to some convex block sequence of $(e_i)_i$.
Finally, Proposition \ref{theoremaboutblocksequences} (iv) yields that (ii) is satisfied.
\end{proof}

\section{Characterizing strongly summing conditional spreading sequences}\label{strongly summing}
As it will become clear in the next section, within the class of conditional spreading sequences there is a distinction between those that are  strongly summing and those that are not. When proving various results, each case may have to be treated separately and sometimes they satisfy different properties. In this relatively brief section we prove a useful criterion for deciding when a given conditional spreading sequence is strongly summing. The concept of a strongly summing sequence is due to H. P. Rosenthal and it first appeared in \cite[Definition 1.1]{R}
\begin{dfn}
A Schauder basic sequence $(e_i)_i$ is called  strongly summing, if for every sequence of real numbers $(a_i)_i$ such that the sequence $(\|\sum_{i=1}^na_ie_i\|)_n$ is bounded, the real series $\sum_ia_i$ is convergent.\label{stronglysumming}
\end{dfn}

The following is proved in \cite[Theorem 1.1]{R}.
\begin{thm}\label{rosenthalstronglysumming}
Let $(x_i)_i$ be a non-trivial weak Cauchy sequence in some Banach space. Then one of the following holds.
\begin{itemize}

\item[(i)] There exists a subsequence of $(x_i)_i$ that is strongly summing.

\item[(ii)] There exists a convex block sequence of $(x_i)_i$ that is equivalent to the summing basis of $c_0$.

\end{itemize}
\end{thm}

\begin{lem}
\label{xdoublestartrailsc0}
Let $X$ be a Banach space with a strongly summing conditional spreading basis $(e_i)_i$ and let $x^{**}\in X^{**}$. If the series $\sum_ix^{**}(e_i^*)e_i$ does not converge in norm, then there exists a strictly increasing sequence of natural numbers $(n_i)_i$, such that if $y_i = \sum_{j=1}^{n_i}x^{**}(e_j^*)e_j$ for all $i\inn$, we have that $(y_i)_i$ is equivalent to the summing basis of $c_0$.
\end{lem}

\begin{proof}
Since $(e_i)_i$ is strongly summing we have that the series $\sum_ix^{**}(e_i^*)$ is convergent. Combining this with the fact that the series $\sum_ix^{**}(e_i^*)e_i$ does not converge in norm, we may choose a strictly increasing sequence of natural numbers $(n_i)_i$, such that if $x_i = \sum_{j=n_{i-1}+1}^{n_i}x^{**}(e_j^*)e_j$ then $(x_i)_i$ is seminormalized and $\sum_i|s(x_i)| < \infty$. Then, by Proposition \ref{summingzero}, the sequence $(y_i)_i$ with $y_i = x_i - s(x_i)e_{\min\supp x_i}$ is unconditional. As the sequences $(x_i)_i$ and $(y_i)_i$  are equivalent (or at least, they have equivalent tails), $(x_i)_i$ must be unconditional as well. Since the sequence $(\|\sum_{j=1}^ix_i\|)_i$ is bounded we conclude that  $(x_i)_i$ is equivalent to the unit vector basis of $c_0$ and hence, if $y_i = \sum_{j=1}^ix_i = \sum_{j=1}^{n_i}x^{**}(e_j^*)e_j$, we have that $(y_i)_i$ is equivalent to the summing basis of $c_0$.
\end{proof}


\begin{lem}\label{dualdnstar}
Let $X$ be a Banach space with a conditional spreading basis $(e_i)_i$ and let $(z_i)_i$ be its convex block homogeneous part. Let $s_{[1,i]} = \sum_{j=1}^ie_j^*$ and $\overline{s}_{[1,i]} = \sum_{j=1}^iz_j^*$ for all $i\inn$.
Then, there exist a sequence of natural numbers $(n_i)_i$ so that for every strictly increasing sequence of natural numbers $(k_i)_i$ with $k_i - k_{i-1} \geqslant n_i$ for all $i\inn$ (where $k_0 = 0$),
the sequence $(s_{[1,k_i]})_i$ is equivalent to $(\overline{s}_{[1,i]})_i$. In particular, every subsequence of $(s_{[1,i]})_i$ has a further subsequence equivalent to $(\overline{s}_{[1,i]})_i$.
\end{lem}

\begin{proof}
Let $(n_i)_i$ be the sequence provided by Lemma \ref{Zxembeds} and let $(k_i)_i$ be a strictly increasing sequence of natural numbers with $k_i - k_{i-1} \geqslant n_i$ for all $i\inn$.
Set $I_i = \{k_{i-1}+1,\ldots,k_i\}$, $x_i = (1/\#I_i)\sum_{j\in s_i}e_j$ and $x_i^* = \sum_{j\in I_i}e_j^*$  for all $\inn$.
By the choice of the sequence $\bar I = (I_i)_i$ we have that $(x_i)_i$ is equivalent to $(z_i)_i$ and the map $P:X\rightarrow X$ with $P_{\bar{I}}x = \sum_{i=1}^\infty s_{I_i}(x)x_i$ is a bounded linear projection.
Observe that $P_{\bar{I}}^*x_i^* = x_i^*$ for all $i\inn$ and that $x^*_i(x_j) = \de_{i,j}$ for all $i,j\inn$.

The above imply that $(x_i^*)_i$ is equivalent to $(z_i^*)_i$. Finally, observe that $\sum_{j=1}^ix_j^* = s_{[1,k_i]}$ for all $i\inn$.
\end{proof}

The following is the main result of this section.

\begin{prp}\label{equivalentfornotsumming}
Let $X$ be a Banach space with a conditional spreading basis $(e_i)_i$ and let $(z_i)_i$ be its convex block homogeneous part. The following assertions are equivalent.
\begin{itemize}

\item[(i)] The space $X$ does not embed into a space with an unconditional basis.

\item[(ii)] The sequence $(z_i)_i$ is not equivalent to the summing basis of $c_0$.

\item[(iii)] The basis $(e_i)_i$ is strongly summing.

\item[(iv)] The sequence $(s_{[1,n]})_n$ is weak Cauchy.

\item[(v)] For every sequence of real numbers $(a_i)_i$ so that $(\sum_{i=1}^na_ie_i)_n$ is bounded, the series $\sum_ia_ie_i$ is weak Cauchy.


\end{itemize}
\end{prp}

\begin{proof}
(i)$\Rightarrow$(ii): If $(z_i)_i$ is equivalent to the summing basis of $c_0$,
then by Remark \ref{equivalentmaxnorm} we have that $X$ embeds into $U\oplus c_0$ which has an unconditional basis.

(ii)$\Rightarrow$(iii): If $(z_i)_i$ is not equivalent to the summing basis of $c_0$, Proposition \ref{theoremaboutblocksequences} (iv) yields that no convex block sequence of $(e_i)_i$ can be equivalent to the summing basis of $c_0$.
The spreading property of $(e_i)_i$ and Theorem \ref{rosenthalstronglysumming} yield that $(e_i)_i$ is strongly summing.

(iii)$\Rightarrow$(i): Assume that the basis $(e_i)_i$ is strongly summing and $X$ embeds into a space with an unconditional basis. It is well known that a non-trivial weak Cauchy sequence in a space with an unconditional basis has a convex block sequence equivalent to the summing basis of $c_0$. Hence, this is true for $(e_i)_i$. It is straightforward to check that if a sequence, in this case $(e_i)_i$, has a convex block sequence equivalent to the summing basis of $c_0$, then it cannot be strongly summing.

(iv)$\Rightarrow$(ii): Assume that $(s_{[1,i]})_i$ is weak Cauchy. If $(z_i)_i$ is equivalent to the summing basis of $c_0$, Lemma \ref{dualdnstar} implies that $(s_{[1,i]})_i$ has a subsequence equivalent to the unit vector basis of $\ell_1$
and it cannot be weak Cauchy.

(iii)$\Rightarrow$(iv): Assume that the basis $(e_i)_i$ is strongly summing, we will show that $(s_{[1,i]})_i$ is weak Cauchy. Let $x^{**}\in X^{**}$.
The fact that $(e_i)_i$ is strongly summing implies that the series $\sum_jx^{**}(e_j^*)$ is convergent. Observe that $x^{**}(s_{[1,i]}) = \sum_{j=1}^ix^{**}(e_j^*)$ for all $i\inn$ and therefore the sequence $(x^{**}(s_{[1,i]}))_i$ is convergent.

(v)$\Rightarrow$(iv): Let $x^{**}\in X^{**}$ and set $a_i = x^{**}(e_i^*)$. Then $(x^{**}(s_{[1,i]}))_n = (s(\sum_{i=1}^na_ie_i))_n$, which by (v) is a convergent sequence.

(iii)$\Rightarrow$(v): If $\sum_ia_ie_i$ is convergent in norm then there is nothing more to prove. Otherwise, it is sufficient to show that if $(n_k)_k$ and $(m_k)_k$ are strictly increasing, then they have subsequences $(n_k')_k$ and $(m_k')_k$ so that the sequence $(\sum_{i=m_k'+1}^{n_k'}a_ie_i)_k$ is weak null.
Combining (iii) with Lemma \ref{xdoublestartrailsc0}, choose $(n_k')_k$ and $(m_k')_k$ so that $m_k' < n_k'$, $\sum_k|\sum_{i=m_k'+1}^{n_k'}a_i|<\infty$
and both  $(\sum_{i=1}^{m_k'}a_ie_i)_k$ and $(\sum_{i=1}^{n_k'}a_ie_i)_k$ are equivalent to the summing basis of $c_0$. We conclude that the sequence $x_k = \sum_{i=m_k'+1}^{n_k'}a_ie_i$ is unconditional and weak Cauchy, i.e. it is weakly null.
\end{proof}

\section{Non-trivial weak Cauchy sequences in spaces with conditional spreading bases.}\label{non-trivial weak Cauchy}
In this section we study the behavior of non-trivial weak Cauchy sequences in a space with a conditional and spreading basis. As it turns out, such sequences always have convex block sequence that are very well behaved. The main result is the following theorem that we prove in several steps.

\begin{thm}
\label{maintheoremweakcauchy}
 Let $X$ be a Banach space with a conditional spreading Schauder basis $(e_i)_i$ and let $(z_i)_i$ be its convex block homogeneous part. Let also $(x_i)_i$ be a non-trivial weak Cauchy sequence in $X$. The following statements hold.
 \begin{itemize}
  \item[(i)] The sequence $(x_i)_i$ has a convex block sequence $(w_i)_i$ that is either equivalent to the summing basis of $c_0$ or equivalent to $(z_i)_i$.
  \item[(ii)] The sequence $(x_i)_i$ has a convex block sequence $(w_i)_i$ the closed linear span of which is complemented in $X$.
 \end{itemize}
\end{thm}

\begin{lem}\label{weaklycauchyminusvector}
Let $X$ be a Banach space and $(x_i)_i$ be a Schauder basic sequence in $X$ so that the summing functional $s(\sum_ia_ix_i) = \sum_ia_i$ is bounded on the space spanned by $(x_i)_i$.
If $x\notin [(x_i)_i]$ then the sequence $(x_i - x)_i$ is equivalent to $(x_i)_i$.
\end{lem}

\begin{proof}
If $\delta = \dist(x,[(x_i)_i])$, then the Hahn-Banach theorem yields that for a sequence of scalars $a_1,\ldots,a_n$ we have
$$\left(\frac{\delta}{\|x\|+\delta}\right)\left\|\sum_{i=1}^na_ix_i\right\|\leqslant \left\|\sum_{i=1}^na_i(x_i-x)\right\|.$$
On the other hand,
$$\left\|\sum_{i=1}^na_i(x_i-x)\right\| \leqslant \left\|\sum_{i=1}^na_ix_i\right\| + \left|\sum_{i=1}^na_i\right|\|x\| \leqslant \left(1+\|s\|\|x\|\right)\left\|\sum_{i=1}^na_ix_i\right\|.$$
\end{proof}


\begin{proof}[Proof of Theorem \ref{maintheoremweakcauchy} (i)]
Assume first that the basis $(e_i)_i$ is not strongly summing. Proposition \ref{equivalentfornotsumming} yields that $X$ embeds into a space with an unconditional basis
and therefore there exists a convex block sequence of $(x_i)_i$ that is equivalent to the summing basis of $c_0$. Assume now that the basis $(e_i)_i$ is strongly summing and let $x^{**}$ be the weak star limit of $(x_i)_i$. We distinguish two cases.

{\em Case 1:} The series $\sum_ix^{**}(e_i^*)e_i$ converges in norm.
Take the vector $x =  \sum_{i=1}^\infty x^{**}(e_i^*)e_i$. Using Lemma \ref{weaklycauchyminusvector} we may assume that $(x_i)_i$ is equivalent to $(x_i - x)_i$. Observe that $(x_i - x)_i$ is point-wise null, with respect to the basis $(e_i)_i$,
and hence we may assume that $(x_i - x)_i$ is a block sequence. If $(x_i - x)_i$ had an unconditional subsequence then, since it is weak Cauchy, it would have to be weakly null, which is absurd.
Proposition \ref{theoremaboutblocksequences} (v) yields that there exists a convex block sequence of $(x_i - x)_i$, hence also of $(x_i)_i$, that is equivalent to $(z_i)_i$. 

{\em Case 2:} The series $\sum_ix^{**}(e_i^*)e_i$ does not converge in norm.
 Using Lemma \ref{xdoublestartrailsc0}, choose a strictly increasing sequence of natural numbers $(n_i)_i$, such that if $y_i = \sum_{j=1}^{n_i}x^{**}(e_j^*)e_j$ for all $i\inn$, then $(y_i)_i$ is equivalent to the summing basis of $c_0$
 and set $w_i = x_i - y_i$. Observe that $(w_i)_i$ is weak Cauchy and that it is point-wise null, with respect to the basis $(e_i)_i$. Passing to a subsequence, we have that $(w_i)_i$ is equivalent to a block sequence.
 If $(w_i)_i$ has an unconditional subsequence then, since it is weak Cauchy, it is weakly null. Mazur's Theorem implies that there exists a convex block sequence of $(w_i)_i$ that converges to zero in norm
 which further yields that there exists a convex block sequence of $(x_i)_i$ that is equivalent to the summing basis of $c_0$. If $(w_i)_i$ does not have an unconditional subsequence,
 Proposition \ref{theoremaboutblocksequences} (v) yields that there exists a convex block sequence $(w_i^\prime)_i$ of $(w_i)_i$ that is equivalent to $(z_i)_i$.

 If $w_i^\prime = \sum_{j\in F_i}a_jw_j$ with $\sum_{j\in F_i}a_j = 1$, set $x_i^\prime = \sum_{j\in F_i}a_jx_j$ and $y_i^\prime = \sum_{j\in F_i}a_jy_j$ for all $j\inn$. Since $(x_i^\prime)_i$ is non-trivial weak Cauchy,
 we may assume that it dominates the summing basis of $c_0$ and, of course, the same is true for $(z_i)_i$. Combining the above with the fact that $(y_i^\prime)_i$ is equivalent to the summing basis of $c_0$
 and $(x_i^\prime - y_i^\prime)_i$ is equivalent to $(z_i)_i$, a simple argument yields that $(x_i^\prime)_i$ is equivalent to $(z_i)_i$.
\end{proof}

\begin{lem}\label{non summing tails are positive}
Let $X$ be a Banach space with a conditional spreading basis $(e_i)_i$ and let also $(x_i)_i$ be a non-trivial weak Cauchy sequence in $X$. If $(x_i)_i$ has no convex block sequence equivalent to the summing basis of $c_0$,
then $\liminf_i\lim_k|s_{(i,+\infty)}(x_k)| > 0.$
\end{lem}

\begin{proof}
Assume that the conclusion is false, i.e.
\begin{equation}
\lim_i\lim_k|s_{(i,+\infty)}(x_k)| = 0.\label{summing tails go to zero and sonic is blue}
\end{equation}
If $x^{**} = w^*$-$\lim_kx_k$, set $y_m = \sum_{i=1}^mx^{**}(e_i^*)e_i$ for all $m$. Note that for all $m$ $\lim_k\|P_{[1,m]}x_k - y_m\| = 0$, which in conjunction with \eqref{summing tails go to zero and sonic is blue} and a sliding hump argument yields that
there exists subsequences $(x_{k_n})_k$ and $(y_{m_n})_n$ so that $(x_{k_n} - y_{m_n})_n$ is point-wise null (with respect to $(e_i)_i$) and $\lim_ns(x_{k_n} - y_{m_n}) = 0$. By Lemma \ref{xdoublestartrailsc0} we may assume that $(y_{m_n})$ is
either equivalent to the summing basis of $c_0$, or norm-convergent. In particular it is weak Cauchy.

Observe that the norm of $(x_{k_n} - y_{m_n})_n$ is eventually bounded from below, otherwise we would obtain that $(x_k)_k$ has a norm convergent subsequence or a subsequence equivalent to the summing basis of $c_0$, which is absurd. Hence, passing to a subsequence, $(x_{k_n} - y_{m_n})_n$ is
equivalent to a weak Cauchy seminormalized block sequence $(w_n)_n$ of $(e_i)_i$ with $s(w_n) = 0$ for all $n\inn$. By Proposition \ref{summingzero}, $(w_n)_n$ is unconditional and, being weak Cauchy, it is weakly null. By Mazur's theorem
we conclude that $(x_k)_k$ has a convex block sequence equivalent to a convex block sequence of $(y_m)$, i.e. equivalent to the summing basis of $c_0$, which is absurd.
\end{proof}

Although the following basically proves Theorem \ref{maintheoremweakcauchy} (ii), it also provides further information about the kernel of the associated projection in certain cases. Its full statement will be required in the sequel.

\begin{prp}
\label{ntwc has complemented subseq}
Let $X$ be a Banach space with a strongly summing conditional spreading basis $(e_i)_i$ and let $(z_i)_i$ be its convex block homogeneous part. Then, every sequence $(x_i)_i$ in $X$ that is equivalent to $(z_i)$ has a subsequence the closed linear span $W$ of which is complemented in $X$. Furthermore,  there exists a sequence $\bar I = (I_i)_i$ of consecutive intervals of $\N$ so that if $Q:X\to X$ is the corresponding projection onto $W$ and $P_{\bar I}:X\to X$ is defined by $P_{\bar I}(x) =  \sum_{i=1}^\infty s_{I_i}(x)((1/\#I_i)\sum_{j\in I_i}e_j)$, then $\ker Q$ is isomorphic to $\ker P_{\bar I}$.
\end{prp}

\begin{proof}

The idea of the proof is the following. We will pass to a subsequence of $(x_n)_n$, again denoted by $(x_n)_n$, find a sequence of consecutive intervals $(I_n)_n$ of the natural numbers with $\#I_n\rightarrow\infty$, a vector $y_0$ in $X\setminus[(x_n)_n]$ (which also means $y_0\in X\setminus[(x_n - y_0)_n]$) and a sequence $(\tilde{w}_n)_n$ so that the following hold:
\begin{itemize}
 
 \item[(i)] there is a non-zero scalar $\al$ so that $$\sum_{n=1}^\infty\left\|\al \tilde{w}_n - (x_n - y_0)\right\| <\infty,$$
 
 \item[(ii)] for all natural numbers $m$, $n$ we have $s_{I_m}(\tilde{w}_n) = \de_{m,n}$.
 
\end{itemize}
We shall first use the above to conclude the proof and describe the construction of the ingredients later. By (i), perhaps omitting the first few terms of each sequence, there is an invertible operator $A:X\rightarrow X$ with $A\tilde{w}_n = x_n - y_0$ for all $n$. We conclude by Lemma \ref{weaklycauchyminusvector} that $(\tilde{w}_n)_n$  is equivalent to $(z_n)_n$ and let $R:[(z_n)_n]\to[(\tilde{w}_n)_n]$ be the map witnessing this fact. By  Proposition \ref{averagingprojection} the map $P_{\tilde I}x = \sum_{i=1}^\infty s_{I_i}(x)((1/\#I_i)\sum_{j\in I_i}e_j)$ is a bounded linear projection and if $v_i = (1/\#I_i)\sum_{j\in I_i}e_j$ for all $i\in\N$ then by \eqref{what is the convex block homogeneous part exactly formula} the map $S:[(v_i)_i]\to[(z_i)_i]$ with $Sv_i = z_i$ is bounded. We conclude that the map $Px = \sum_{n=1}^\infty s_{I_n}(x)\tilde{w}_n$ is a bounded linear projection, as $P = RSP_{\bar I}$ and clearly $\ker P = \ker P_{\bar I}$. If $\tilde Q = APA^{-1}$, then $\tilde Q$ is a projection onto $[(x_k - y_0)_k]$ the kernel of which is isomorphic to $\ker P_{\bar I}$. As $y_0$ is not in $\tilde Q[X]$, by the Hahn-Banach Theorem, we may choose a norm-one linear functional $f$ on $X$ with $\tilde Q[X]\subset\ker f$. Define $V_0 = \ker f\cap\ker \tilde Q$ and $W_0 = \tilde Q[X] + \langle\{y_0\}\rangle$. Note that $W = [(x_k)_k]$ is of co-dimension one in $W_0$. We deduce that setting $V = V_0 + \langle\{y_0\}\rangle$ the spaces $W$ and $V$ are complementary and their sum is the space $X$. It is also immediate that $V$ isomorphic to $\ker \tilde Q$ which isomorphic to $\ker P_{\bar I}$.

We now proceed to present how the aforementioned components are constructed. Fix decreasing sequences of positive real numbers $(\de_n)_n$, $(\e_i)_i$ so that $\sum_n\de_n < 1/4$ and
\begin{equation}\label{ntwc has complemented cb equation chuchu}
\sum_{n=1}^\infty\sum_{F\in[n,\infty)^{<\infty}}\prod_{i\in F}\e_i < 1/2,
\end{equation}
where $[n,\infty)^{<\infty}$ denotes the set of all finite subsets of the
natural numbers with $\min F\geqslant n$. Use Lemma \ref{non summing tails are positive} to find a positive real number $\al$ so that $\liminf_i\lim_k|s_{(i,\infty)}(x_k)| = \al$. By perhaps multiplying all terms of the sequence $(x_n)_n$ with $-1$,
we may assume that $\liminf_i\lim_k|s_{(i,\infty)}(x_k) - \al| = 0$.
If $x^{**} = w^*-\lim_kx_k$, since $(e_i)_i$ is strongly summing, we may choose a sequence $(j_n)_n$ so that for all $j_n\leqslant k\leqslant m$ we have
\begin{equation}\label{ntwc has complemented cb equation pewpew}
\left|\sum_{i=k}^mx^{**}\left(e_i^*\right)\right| < \frac{\al\e_n}{2}.
\end{equation}
Choose a strictly increasing sequence of  natural numbers $(i_n)_n$, with $i_n\geqslant j_n$ and $\lim_n(i_{n+1}-i_n) = \infty$,  and a subsequence of $(x_n)_n$, again denoted by $x_n$, so that for all natural numbers $n\leqslant m$
\begin{equation}\label{ntwc has complemented cb equation powpow}
\left|s_{(i_n,+\infty)}(x_m) - \al\right| < \frac{\al\de_{n}}{2}.
\end{equation}
We can simultaneously choose subsequences of $(i_n)_n$ and $(x_n)_n$ that satisfy:
\begin{subequations}
 \begin{equation}\label{ntwc has complemented cb equation boom}
  \left|s_{(i_n,i_{n+1}]}(x_n) - \al\right| < \al\de_{n},
 \end{equation}
  \begin{equation}\label{ntwc has complemented cb equation broom}
  \left\|P_{[1,i_n]}x_n - \sum_{i=1}^{i_n}x^{**}(e_i^*)e_i\right\| < \de_{n},\text{ and}
 \end{equation}
   \begin{equation}\label{ntwc has complemented cb equation groom}
  \left\|P_{(i_{n+1},+\infty)}x_n\right\| < \de_{n}.
 \end{equation}
\end{subequations}
For all $n\inn$ define $I_n = (i_n,i_{n+1}]$ and
$$\tilde{x}_n = \sum_{i=1}^{i_n}x^{**}(e_i^*)e_i + P_{I_n}x_{n}.$$
Define
$$y_0 = \sum_{n=1}^\infty\left(\sum_{i\in I_n}x^{**}(e_i^*)\right)\frac{1}{s_{I_n}\left(P_{I_n}\tilde{x}_n\right)} P_{I_n}\tilde{x}_n.$$
Note that $P_{I_n}\tilde{x}_n = P_{I_n}x_n$ and hence by \eqref{ntwc has complemented cb equation pewpew} and \eqref{ntwc has complemented cb equation boom} $y_0$ is well defined.
Choose $n_0$ so that $y_0$ is not in the closed linear span of $(x_n)_{n\geqslant n_0}$ and hence,
by Lemma \ref{weaklycauchyminusvector},
$(x_n - y_0)_{n\geqslant n_0}$ is equivalent to $(x_n)_n$, i.e. to $(z_n)_n$. To simplify notation we shall assume that $n_0 = 1$.

Since $P_{I_n}\tilde{x}_n = P_{I_n}x_n$, by \eqref{ntwc has complemented cb equation pewpew} and \eqref{ntwc has complemented cb equation boom} we obtain:
\begin{equation}\label{ntwc has complemented cb equation shuriken}
\left|s_{I_n}(\tilde{x}_n - y_0) - \al\right| = \left|s_{I_n}(x_n) - \sum_{i\in I_n}x^{**}(e_i^*) - \al\right| < \al(\de_n + \e_n/2),
\end{equation}
in particular $|s_{I_n}(\tilde{x}_n - y_0)| > \al/2$, hence for all $n\inn$ we may define $w_n = (s_{I_n}(\tilde{x}_n - y_0))^{-1}(\tilde{x}_n - y_0)$. It is straightforward to check the following:
\begin{subequations}
\begin{equation}\label{ntwc has complemented cb equation banana}
s_{I_n}(w_m) = \left\{
\begin{array}{ll}
0& \text{if } m<n \;\text{and}\\
1  & \text{if } m=n.
\end{array}
\right.
\end{equation}
Also if $m>n$, $s_{I_n}(w_m) = (s_{I_n}(\tilde{x}_n - y_0))^{-1}\sum_{i\in I_m}x^{**}(e_i^*)$. Hence, by $|s_{I_n}(\tilde{x}_n - y_0)| > \al/2$ and \eqref{ntwc has complemented cb equation pewpew} for $m>n$
\begin{equation}\label{ntwc has complemented cb equation katana}
|s_{I_m}(w_n)| < \e_n.
\end{equation}
\end{subequations}
For fixed $n$ recursively define a sequence of scalars $(c_j^n)_{j\geqslant n+1}$ as follows: $c_{n+1}^n = -s_{I_{n+1}}(w_{n})$ and $c_{j+1}^n = -s_{I_{j+1}}(w_n) - \sum_{i=n+1}^jc_i^ns_{I_{j+1}}(w_i)$. Using \eqref{ntwc has complemented cb equation katana},
keeping $n$ fixed and an induction on $j\geqslant n+1$ we obtain:
\begin{equation*}
|c_j^n| < \sum_{\substack{F\subset [n+1,j]\\ \max F = j}} \prod_{i\in F}\e_i
\end{equation*}
and therefore
\begin{equation}\label{ntwc has complemented cb equation kumquat}
\sum_{j=n+1}^\infty|c_j^n| < \sum_{F\in [n+1,\infty)^{<\infty}} \prod_{i\in F}\e_i
\end{equation}
which, in conjunction with \eqref{ntwc has complemented cb equation chuchu}, yields that for $n\inn$ the vector $$\tilde{w}_n = w_n + \sum_{j=n+1}^\infty c_j^nw_j$$ is well defined.

It remains to show that (i) and (ii) are satisfied to complete the proof. The fact that (ii) holds is an easy consequence of \eqref{ntwc has complemented cb equation banana} and the choice of the sequences $(c_{j}^n)_{j\geqslant n+1}$ for $n\inn$.
By \eqref{ntwc has complemented cb equation kumquat} and \eqref{ntwc has complemented cb equation chuchu} we obtain $\sum_{n=1}^\infty\|\tilde{w}_n - w_n\| <\infty$, therefore it remains to observe that $\sum_{n=1}^\infty\|\al w_n - (x_n - y_0)\| <\infty$.
Indeed, for $n\inn$ \eqref{ntwc has complemented cb equation shuriken} yields that $\|\al w_n - (\tilde{x}_n - y_0)\| < 2\de_n + \e_n$ and by \eqref{ntwc has complemented cb equation broom} and \eqref{ntwc has complemented cb equation groom}
$\|\tilde{x}_n - x_n\| < 2\de_n$.
\end{proof}

\begin{proof}[Proof of Theorem \ref{maintheoremweakcauchy} (ii)]
If $(x_n)_n$ has a convex block sequence equivalent to the summing basis of $c_0$, then the result follows from the well known fact that $c_0$ is separably injective. If this is not the case, let $(z_i)_i$ be the convex block homogeneous part of $(e_i)_i$,
and apply Theorem \ref{maintheoremweakcauchy} (i) to find a convex block sequence of $(w_i)_i$ of $(x_i)_i$ that is equivalent to $(z_i)_i$. Note that, in this case, $(z_i)_i$, being equivalent to $(w_i)_i$,  cannot be equivalent to the summing basis of $c_0$, therefore, by Proposition \ref{equivalentfornotsumming}, the basis $(e_i)_i$ of $X$ must be strongly summing. By proposition \ref{ntwc has complemented subseq}, $(w_i)_i$ has a subsequence the closed linear span $W$ of which is complemented in $X$. 
\end{proof}

\section{Complemented subspaces of spaces with conditional spreading bases}\label{complemented subspaces}

The main results of this section are the following two theorems. The first one characterizes strongly summing spreading bases with respect to squares of spaces and the second one provides information about arbitrary decompositions of a space with a strongly summing spreading basis. Their proofs are based on certain projections and Theorem \ref{maintheoremweakcauchy} (ii), in fact the more precise statement of Proposition \ref{ntwc has complemented subseq}. This section concludes with a conversation around the question if spaces with a convex block homogeneous basis are primary.

\begin{thm}
\label{squares}
Let $X$ be a Banach space with a spreading basis $(e_i)_i$, let $(z_i)_i$ be its convex block homogeneous part, and set $Z = [(z_i)_i]$. The following hold.
\begin{itemize}
 \item[(i)] If $(e_i)_i$ is strongly summing, then $Z\oplus Z$ does not embed into $X$. In particular, $X\oplus X$ does not embed into $X$.
 \item[(ii)] The basis $(e_i)_i$ is not strongly summing if and only if $X$ is isomorphic to $X\oplus X$.
\end{itemize}
\end{thm}

\begin{thm}
\label{there can be only one}
Let $X$ be a Banach space with a  strongly summing conditional spreading basis $(e_i)_i$, let $(z_i)_i$ be its convex block homogeneous part, and $Z = [(z_i)_i]$. If $X = V\oplus W$, then exactly one of the spaces $V$ and $W$ contains a subspace $\tilde Z$ that is isomorphic to $Z$. Furthermore, $\tilde Z$ is complemented in the whole space.
\end{thm}

\subsection*{UFDD's and skipped unconditionality}
A tool used to prove the above results are projections like the one from Proposition \ref{averagingprojection} and their kernels which are studied in this subsection.

If $N = \{j_1 < j_2 < \cdots\} \subset \N$, we shall say that a sequence of successive intervals $(E_j)_{j\in N}$ of $\N$ is skipped, if $\max E_{j_i} + 1 < \min E_{j_{i+1}}$ for all $i$.

\begin{ntt}
Let $X$ be a Banach space with a conditional spreading basis $(e_i)_i$ and let $(d_i)_i$ be the difference basis of $X$.
\begin{itemize}
 
 \item[(i)] If $E$ is an interval of $\N$ we denote by $X_E$ the linear span of the vectors $(d_i)_{i\in E}$.
 
 \item[(ii)] If $\bar E = (E_j)_{j\in N}$ is a sequence of skipped intervals of $\N$, where $N\subset \N$, we denote by $X_{\bar E}$ or $X_{(E_j)_{j\in N}}$ the closed linear span of $\cup_{j\in N} X_{E_j}$.  

 \item[(iii)] If $\bar I = (I_i)_i$ is a sequence of consecutive intervals of $\N$, we denote by $\reflectbox{$'$}\bar I$ the skipped sequence of successive intervals $(\reflectbox{$'$} I_j)_{j\in N}$,
 where $\reflectbox{$'$} I_j = I_j\setminus\{\min I_j\}$ and $N = \{j: \reflectbox{$'$} I_j\neq\varnothing\}$.
 
 \end{itemize}

\end{ntt}

\begin{prp}\label{skipped intervals unconditional}
Let $X$ be a Banach space with a conditional spreading basis $(e_i)_i$ and let $(E_j)_{j\in N}$ be a sequence of skipped intervals of $\N$. Then $(X_{E_j})_{j\in N}$
 is an unconditional finite dimensional decomposition (UFDD) of $X_{(E_j)_{j\in N}}$. If moreover $(e_i)_i$ is 1-spreading, then the FDD $(X_{E_j})_{j\in N}$ is suppression unconditional.
\end{prp}

\begin{proof}
We may assume that $(e_i)_i$ is 1-spreading. The result easily follows by combining Proposition \ref{summingzero} with the fact that any sequence $(x_j)_j$ with $x_j$ in $X_{E_j}$ is a block sequence of $(e_i)_i$ with $s(x_j) = 0$ for all $j\in N$.
\end{proof}

\begin{rmk}\label{skipping with same cardinalities means isomorphic}
Let $X$ be a Banach space with a conditional spreading basis $(e_i)_i$ and let also $(E_j)_j$ and $({F}_j)_j$ be skipped sequences of successive intervals of $\N$ so that $\#E_j = \#{F}_j$. The spreading property of $(e_i)_i$ implies that the spaces $X_{(E_j)_j}$ and $X_{({F}_j)_j}$ are naturally isomorphic through the map $d_i\rightarrow d_{\phi(i)}$, where if $i$ is the $k$-th element of $E_j$, then $\phi(i)$ is the $k$-th element of $F_j$. If furthermore the basis $(e_i)_i$ is 1-spreading and $\min(E_1) = \min(F_1) = 1$ or $1<\min\{\min(E_1),\min(F_1)\}$, then the spaces $X_{(E_j)_j}$ and $X_{({F}_j)_j}$ are naturally isometric. The condition concerning the minima of $E_1$ and $F_1$ is due to $d_1 = e_1$ whereas $d_{i+1} = e_{i+1} - e_i$ for all $i\in\N$.
\end{rmk}

\begin{prp}\label{averagingprojectionkernel}
Let $X$ be a Banach space with a conditional spreading basis $(e_i)_i$. Let also $\bar{I} = (I_i)_i$ be a sequence of consecutive intervals of the natural numbers with $\cup_{i}I_i = [i_0,+\infty)$, for some $i_0\inn$
and let $P_{\bar{I}}$ be the associated averaging projection.
Then, the kernel of $P_{\bar{I}}$ is the space $\langle\{e_1,\ldots,e_{i_0-1}\}\rangle \oplus X_{\reflectbox{\tiny$'$}\bar I}$. In particular, the kernel of $P_{\bar{I}}$ admits a UFDD.
\end{prp}

\begin{proof}
It is easy to see that $X_{\reflectbox{\tiny$'$}I_j}$ is in the kernel of $P$ for all $j$ and an elementary argument yields that the spaces $X_{\reflectbox{\tiny$'$}I_1},\ldots,X_{\reflectbox{\tiny$'$}I_n}$ together with the vectors
$z_j =   (1/\#I_j)\sum_{i\in I_j}e_i$, $j=1,\ldots,n$ and $e_1,\ldots,e_{i_0-1}$ span the space $\langle\{e_i: i\leqslant \max I_n\}\rangle$. The above implies that $\langle\{e_1,\ldots,e_{i_0-1}\}\rangle \oplus X_{\reflectbox{\tiny$'$}\bar I}$ is the entire kernel of $P_{\bar{I}}$. By Proposition \ref{skipped intervals unconditional} the kernel of $P_{\bar{I}}$ admits a UFDD.
\end{proof}

\begin{cor}\label{if they go fast enough they contain all unconditionals}
Let $X$ be a Banach space with a conditional spreading basis $(e_i)_i$. Let also $(E_j)_j$ be a skipped sequence of successive intervals of $\N$ so that $\sup_j\#E_j = \infty$. Then every unconditional sequence $(x_i)_i$ in $X$ has a subsequence that is equivalent to some sequence in $X_{(E_j)_j}$.
\end{cor}

\begin{proof}
By Remark \ref{skipping with same cardinalities means isomorphic} we may assume that $\max E_j + 2 = \min E_{j+1}$ for all $j\inn$ and $\min E_1 = 2$. For each $j$ set $I_j = \{\min E_j-1\}\cup E_j$. Then if $\tilde{z}_i = (1/\#I_i)\sum_{j\in s_{I_i}}e_j$, by Proposition \ref{averagingprojection} the map $P:X\rightarrow X$ with $Px = \sum_is_{I_i}(x)\tilde{z}_i$ is a bounded linear projection and  $\mathrm{ker}P$ is $X_{(E_j)_j}$. Let now $(x_i)_i$ be an unconditional sequence in $X$. Then,  it is either weakly null or it has a subsequence equivalent to the unit vector basis of $\ell_1$.
In the first case $\lim_is(x_i) = 0$ and by perturbing and passing to a subsequence we may assume that $(x_i)_i$ is a block sequence with $s(x_i) = 0$ for all $i\inn$. In the second case by a standard difference argument we may assume the same. Since the basis $(e_i)_i$ is spreading and $\sup_i\#I_i = \infty$,
we may moreover assume that for every $i\inn$ there exists $k_i$ with $\supp x_i \subset I_{k_i}$. This implies that $x_i\in\ker P$ for all $i\inn$ which is the desired result.
\end{proof}


The following result is not used in the paper however it demonstrates that spaces with a spreading Schauder basis that contain isomorphic copies of $\ell_1$ contain further copies of $\ell_1$ that are very well behaved. 

\begin{prp}\label{ell1 is complemented}
Let $X$ be a Banach space with a conditional and spreading basis $(e_i)_i$ containing a subspace $Y$ that is isomorphic to $\ell_1$. Then $Y$ contains a further subspace $W$ that is isomorphic to $\ell_1$ and complemented in $X$.
In particular, whenever a Banach space $X$ with a conditional spreading basis contains a copy of $\ell_1$, then $X$ contains a complemented copy of $\ell_1$.
\end{prp}

\begin{proof}
If $(y_i)_i$ is a sequence in $X$ equivalent to the unit vector basis of $\ell_1$, then by passing to differences and perturbing, we may assume that it is a block sequence with $s(y_i) = 0$ for all $i\inn$. Choose a sequence of
consecutive intervals $(I_i)_i$ of $\N$ so that $\ran y_i \subset I_i$ for all $i\inn$. Then, by Proposition \ref{averagingprojection}, the map $P_{\bar{I}}x = \sum_is_{I_i}(x)/\#I_i\sum_{j\in I_i}e_j$ is a bounded linear projection. By Proposition \ref{averagingprojectionkernel} the kernel of $P_{\bar{I}}$ admits a UFDD. Observe that $(y_i)_i$ is in $\ker P_{\bar I}$. It is well known, and not hard to prove, that a copy of $\ell_1$ in a space with a UFDD contains a further copy of $\ell_1$ complemented in the whole space. That is, there is a copy of $\ell_1$ complemented in the kernel of $P_{\bar{I}}$, and thus complemented in $X$.
\end{proof}

\begin{prp}\label{ZxplusuncFDD}
Let $X$ be a Banach space with a conditional spreading basis $(e_i)_i$, let $(z_i)_i$ be its convex block homogeneous part, and $Z = [(z_i)_i]$. Then there exists a  sequence of natural numbers $(m_i)_i$, so that for every skipped
sequence of successive intervals $(E_j)_j$ of $\N$ with $\#E_j\geqslant m_j$ for all $j$, we have that $X$ is isomorphic to $Z\oplus X_{(E_j)_j}$.
\end{prp}

\begin{proof}
Choose a sequence of natural numbers $(m_i)_i$ satisfying $m_i \geqslant n_i - 1$, where $(n_i)_i$ is the sequence of Lemma \ref{Zxembeds}. Let $(E_j)_j$ be a skipped sequence of successive intervals  of $\N$ with $\#E_j\geqslant m_i$ for all $i$. By Remark \ref{skipping with same cardinalities means isomorphic} we may assume that $\max E_j + 2 = \min E_{j+1}$ for all $j\inn$. For each $j$ set $I_j = \{\min E_j-1\}\cup E_j$. Then $(I_j)_j$ satisfies the assumptions of Lemma \ref{Zxembeds} and if $\tilde{z}_i = (1/\#I_i)\sum_{j\in s_{I_i}}e_j$ then $(\tilde{z}_i)_i$ is equivalent to $(z_i)_i$. By Proposition \ref{averagingprojection} the map $P:X\rightarrow X$ with $Px = \sum_is_{I_i}(x)\tilde{z}_i$
is a bounded linear projection and  $\mathrm{Im}P$ is isomorphic to $Z$ and $\ker P$ is $X_{(E_j)_j}$.
\end{proof}

\begin{rmk}\label{if convex block homogeneous then all UFDDs are absorbed}
If the basis $(e_i)_i$ of a space $X$ is convex block homogeneous, then the sequence $(\tilde{z}_i)_i$ is equivalent to $(z_i)_i$  without imposing any restrictions on the cardinalities of the sets $E_i$. In particular, for every sequence of successive intervals $(E_j)_j$ of $\N$, the space $X$ is isomorphic to $X\oplus X_{(E_j)_j}$.
\end{rmk}

For clarity, we restate what Proposition \ref{ZxplusuncFDD} and Corollary \ref{if they go fast enough they contain all unconditionals} yield in the following.

\begin{prp}\label{uncfddplusZx}
Let $X$ be a Banach space with a conditional spreading basis $(e_i)_i$, let  $(z_i)_i$ be its convex block homogeneous part, and $Z = [(z_i)_i]$. Then, there exists a skipped sequence of successive intervals $(E_j)_j$ of $\N$ so that the following hold.

\begin{itemize}

\item[(i)] The space $X$ is isomorphic to $Z\oplus X_{(E_j)_j}$, i.e. $X$ is the direct sum of a space with a convex block homogeneous conditional spreading basis and a space with a UFDD.

\item[(ii)] Every unconditional sequence $(x_i)_i$ in $X$ has a subsequence equivalent to a sequence in $X_{(E_j)_j}$.

\end{itemize}
\end{prp}

\begin{proof}[Proof of Theorem \ref{squares}]
Let $X$ be a space with a conditional spreading basis $(e_i)_i$, let $(z_i)_i$ be its convex block homogeneous part, and $Z = [(z_i)_i]$. We first prove (i), hence in this case $(e_i)_i$ is strongly summing. Towards a contradiction, assume that there are sequences $(z_i^1)_i$ and $(z_i^2)_i$ in $X$, both equivalent to $(z_i)_i$, so that if $Z_1 = [(z_i^1)_i]$ and $Z_2 = [(z_i^2)_i]$, then $Z_1\cap Z_2 = \{0\}$ and $Z_1 + Z_2$ is a closed subspace of $X$. By Propositions \ref{ntwc has complemented subseq} and \ref{averagingprojectionkernel} there is a subsequence $(w_i)_i$ of $(z_i^1)_i$ and a projection $P:X\rightarrow X$ onto the space $W = [(w_i)_i]$ so that the kernel of $P$ has a UFDD. The open mapping theorem implies that $Z_2$ embeds into $\mathrm{ker}P$. Recall that non-trivial weak Cauchy sequences in spaces with a UFDD have convex block sequences equivalent to the summing basis of $c_0$.  We conclude that $(z_i^2)_i$ has a convex block sequence equivalent to the summing basis of $c_0$, which implies that $(z_i)_i$ is equivalent to the summing basis of $c_0$ and hence, by Proposition \ref{equivalentfornotsumming}, $(e_i)_i$ cannot be strongly summing.

We now prove (ii). By (i), it is enough to show that if $(e_i)_i$ is not strongly summing, then $X$ is isomorphic to $X\oplus X$. By Proposition \ref{equivalentfornotsumming}, the convex block homogeneous part $(z_i)_i$ of $(e_i)_i$ is equivalent to the summing basis of $c_0$. Let $(m_i)_i$ be the sequence provided by Proposition \ref{ZxplusuncFDD} and for $i\inn$ set $m'_i = \max\{m_i,m_{2i-1}, m_{2i}\}$. Choose a skipped sequence of successive intervals $\overline{E} = (E_j)_j$ of $\N$ so that for all $j$, $\#E_{2j-1} = \#E_{2j} = m_j'$. If $\overline{E}_{\mathrm{o}} = (E_{2j-1})_j$ and $\overline{E}_{\mathrm{e}} = (E_{2j})_j$, then by Proposition \ref{ZxplusuncFDD} we conclude that $X$ is isomorphic to $c_0\oplus X_{\overline{E}}$ as well as to $c_0\oplus X_{\overline{E}_{\mathrm{o}}}$. Proposition \ref{skipped intervals unconditional} implies that $X_{\overline{E}} = X_{\overline{E}_{\mathrm{o}}}\oplus X_{\overline{E}_{\mathrm{e}}}$ whereas by Remark \ref{skipping with same cardinalities means isomorphic} we obtain $X_{\overline{E}_{\mathrm{o}}}\simeq X_{\overline{E}_{\mathrm{e}}}$. We conclude:
\begin{eqnarray*}
X &\simeq& c_0\oplus X_{\overline{E}} \simeq c_0\oplus\left(X_{\overline{E}_{\mathrm{o}}}\oplus X_{\overline{E}_{\mathrm{e}}}\right)\simeq \left(c_0\oplus c_0\right)\oplus\left(X_{\overline{E}_{\mathrm{o}}}\oplus X_{\overline{E}_{\mathrm{o}}}\right)\\
&\simeq& \left(c_0\oplus X_{\overline{E}_{\mathrm{o}}} \right)\oplus\left(c_0\oplus X_{\overline{E}_{\mathrm{o}}}\right) \simeq X\oplus X.
\end{eqnarray*}
\end{proof}



\begin{proof}[Proof of Theorem \ref{there can be only one}]
We shall prove that $Z$ embed isomorphically into at least one of the spaces $V$ or $W$. The rest of the statement follows from Proposition \ref{ntwc has complemented subseq} and Theorem \ref{squares} (i).
Let $P:X\rightarrow X$ be the projection with $\mathrm{Im}P = V$ and $\mathrm{ker}P = W$. Note that either $(Pe_i)_i$ or $(e_i - Pe_i)_i$ has be non-trivial weak Cauchy and we shall assume that the first one holds. By Theorem \ref{maintheoremweakcauchy} (i), $(e_i)_i$ has  a convex block sequence $(w_i)_i$, so that $(Pw_i)_i$ is either equivalent to $(z_i)_i$, or to the summing basis of $c_0$. If the first one holds, then $Z$ embeds into $V$ and there is nothing left to prove.

Otherwise, $(Pw_i)_i$ is equivalent to the summing basis of $c_0$. This implies that $(w_i - Pw_i)_i$ has no convex block sequence equivalent to the summing basis of $c_0$. Indeed, assume that there is a convex block sequence $(\tilde{w}_i)_i$ of $(e_i)_i$ so that both $(P\tilde{w}_i)_i$ as well as $(\tilde{w}_i - P\tilde{w}_i)$ are equivalent to the summing basis of $c_0$. Then $(\tilde{w}_i)_i$, being non-trivial weak Cauchy, has to be equivalent to the summing basis of $c_0$ as well. By \eqref{what is the convex block homogeneous part exactly formula} it follows that $(z_i)_i$ is dominated by $(\tilde{w}_i)$ and hence it is equivalent to the summing basis of $c_0$ as well. Proposition \ref{equivalentfornotsumming} (ii)$\Leftrightarrow$(iii) yields a contradiction.

Note that $(w_i - Pw_i)_i$ is non-trivial weak Cauchy. If this were not the case and $w$-$\lim_i(w_i - Pw_i) = x_0$, then $((w_i - x_0) - (Pw_i))_i$ is $w$-null. By Mazur's theorem, $(w_i - x_0)_i$ has a convex block sequence $(\tilde{w}_i - x_0)_i$ equivalent to $(P\tilde{w}_i)_i$, i.e. equivalent to the summing basis of $c_0$. It follows that $(w_i)_i$ has a convex block sequence equivalent to the summing basis of $c_0$, which we showed is not possible.

We have concluded that $(w_i - Pw_i)_i$ is non-trivial weak Cauchy and it has no convex block sequence equivalent to the summing basis of $c_0$. By Theorem \ref{maintheoremweakcauchy} (i) the space $Z$ embeds into $W$.
\end{proof}

The most important examples of spaces with a convex block homogeneous basis are $c_0$, $\ell_1$, and James space $J$. The first two spaces were shown to be prime by A. Pe\l czy\'nski, i.e. their complemented subspaces are isomorphic to the whole space. This is no longer true for James space, however as it was shown in \cite{C} the space is primary. This means that if $J = X\oplus Y$ then one of the space $X$ or $Y$ is isomorphic to $J$. This can also be shown for the spaces $J_p$, $1<p<\infty$.
\begin{prb}
Let $X$ be a Banach space with a convex block homogeneous basis. Is $X$ primary?
\end{prb}
Note that a space with a conditional spreading basis that is not convex block homogeneous may fail to be primary. Consider the norm on $c_{00}(\N)$ with $\|x\| = \max\{\|x\|_J,\|x\|_p\}$, for some $1<p<2$, and denote its completion by $X$. Then $X\simeq J\oplus V$ where $V$ is a reflexive space with a UFDD that contains $\ell_p$, hence $X$ is not primary.


Before concluding this section we present some results concerning spaces with convex block homogeneous bases that hint towards a positive solution of the above mentioned problem. These results are also of independent interest as they exhibit strong similarity between an arbitrary Banach space with a convex block homogeneous basis and James space.

We introduce some further notation that will make stating and proving the subsequent results possible. If $X$ has a conditional convex block homogeneous basis and $(n_j)_{j\in L}$ is a sequence of natural numbers indexed by a subset $L$ of $\N$ denote by $X_{(n_j)_{j\in L}}$ any space $X_{(E_j)_{j\in L}}$ such that there exists a sequence of skipped intervals $(E_j)_{j\in L}$ with $\#E_j = n_j$ for all $j\in L$. By Remark \ref{skipping with same cardinalities means isomorphic}, the choice of the intervals does not matter. For isometric reasons we can choose $\min(E_1) >1$. We shall also allow some of the $n_j$'s to be zero and in this case by $X_{(n_j)_{j\in L}}$ we refer to the space $X_{(n_j)_{j\in\tilde L}}$ where $\tilde L = \{j\in L: n_j\neq 0\}$. For finite $F\subseteq E\subset \N$ we write $F\preceq E$ if $F$ is an initial segment of $E$.

\begin{lem}
\label{step2}
Let $X$ be a Banach with a 1-convex block homogeneous basis $(e_i)_i$.  If $(E_j)_j$ is a sequence of skipped intervals of $\N$ and $(F_j)_j$ is such that $F_j\preceq E_j$ for all $j\in\N$, then $X_{(F_j)_j}$ is naturally complemented in $X_{(E_j)_j}$. Hence,  if $(n_j)_{j\in L}$ and $(s_j)_{j\in L}$ are sequences of non-negative integers with $s_j\leqslant n_j$, then $X_{(n_j)_{j\in L}}\simeq X_{(n_j-s_j)_{j\in L}}\oplus X_{(s_j)_{j\in L}}$. 
\end{lem}

\begin{proof}
We may assume that $\max(E_j) + 2 = \min(E_{j+1})$ for all $j\in\N$. Define $k_1 = 1$, for $j\geqslant 2$ $k_j = \min(E_j)-1$, and for all $j\in\N$ set $\ell_j = \max(F_j)$ and $m_j = \max(E_j)$. A calculation yields that if $x = \sum_{j=1}^N\sum_{i\in E_j}a_id_i$, then
$$y = \sum_{j=1}^N\sum_{i\in F_j}a_id_i = \sum_{j=1}^n\left(\left(\sum_{k_j\leqslant i <\ell_j}s_{\{i\}}(x)e_i\right)+s_{[\ell_j,m_j]}(x)e_{\ell_j}\right).$$
By 1-convex block homogeneity we obtain
\begin{align*}
\|y\| =& \left\|\sum_{j=1}^n\left(\left(\sum_{k_j\leqslant i <\ell_j}s_{\{i\}}(x)e_i\right)+s_{[\ell_j,m_j]}(x)\left(\frac{1}{\#[\ell_j,m_j]}\sum_{i\in[\ell_j,m_j]}e_i\right)\right)\right\|\\
\leqslant &\|x\|\text{ (by Proposition \ref{inthisnicecaseprojectionhasnormone}).}
\end{align*}
\end{proof}


\begin{prp}
\label{step3}
Let $X$ be a Banach space with a 1-convex block homogeneous basis $(e_i)_i$ and let $(n_j)_j$ and $(m_j)_j$ be unbounded sequences of natural numbers. Then, $X_{(n_j)_j}$ is 256-isomorphic to $X_{(m_j)_j}$.
\end{prp}

\begin{proof}
By induction on $d=1,2,\ldots$ choose non-negative integers $(s_j)_{j=1}^d$, $(t_j)_{j=1}^d$, $(k_j)_{j=1}^d$ so that
\begin{itemize}
 \item[(i)] $1\leqslant k_1<\cdots<k_d$,
 \item[(ii)] for $1\leqslant j\leqslant d$ we have $0\leqslant s_j\leqslant m_{k_j}$ and $0\leqslant t_j\leqslant n_{k_j}$,
 \item[(iii)] if $1\leqslant j\leqslant d$ is not in $\{k_1,\ldots,k_d\}$ then $m_j + s_j = n_j + t_j$, and
 \item[(iv)] if $1\leqslant j\leqslant d$ is equal to $k_i$ for some $i \leqslant j$ then $m_j - s_i + s_j = n_j - t_i + t_j$. 
\end{itemize}
There is no difficulty to the induction so we omit it. We  now perform a decomposition argument. Define $K = \{k_j:j\in\N\}$.
\begin{align*}
X_{(m_j)_{j}} &\simeq X_{(m_j)_{j\notin K}}\oplus X_{(m_j)_{j\in K}} \text{ (by unconditionality)}\\
&= X_{(m_j)_{j\notin K}}\oplus X_{(m_{k_j})_{j\in\N}}\\
& \simeq X_{(m_j)_{j\notin K}} \oplus\left(X_{(m_{k_j}-s_j)_{j\in\N}}\oplus X_{(s_j)_{j\in\N}}\right)\text{ (by Lemma \ref{step2} and (ii))}\\
& \simeq X_{(m_j)_{j\notin K}} \oplus\left(X_{(m_{k_j}-s_j)_{j\in\N}}\oplus\left(X_{(s_j)_{j\in K}}\oplus X_{(s_j)_{j\notin K}}\right)\right)\\
& = \left(X_{(m_j)_{j\notin K}} \oplus X_{(s_j)_{j\notin K}}\right)\oplus \left(X_{(m_{k_j}-s_j)_{j\in\N}}\oplus X_{(s_{k_j})_{j\in\N}}\right)\\
& \simeq X_{(m_j+s_j)_{j\notin K}} \oplus X_{(m_{k_j} - s_j + s_{k_j})_{j\in\N}}.
\end{align*}
An identical argument yields $X_{(n_j)_{j}}\simeq X_{(n_j+t_j)_{j\notin K}} \oplus X_{(n_{k_j} - s_j + t_{k_j})_{j\in\N}}$ and by (iii) and (iv) we obtain
$$X_{(m_j+s_j)_{j\notin K}} \oplus X_{(m_{k_j} - s_j + s_{k_j})_{j\in\N}} = X_{(n_j+t_j)_{j\notin K}} \oplus X_{(n_{k_j} - s_j + t_{k_j})_{j\in\N}},$$
i.e $X_{(m_j)_{j}} \simeq X_{(n_j)_{j}}$. If all direct sums are taken with the max-norm, then it follows that the spaces $X_{(m_j)_{j\in\N}}$ and $X_{(n_j)_{j\in\N}}$ are 256-isomorphic.
\end{proof}

\begin{rmk}
\label{ifboundedunconditionalpart}
It can also be shown that for a bounded sequence $(m_j)_j$ the space $X_{(m_j)_j}$ is isomorphic to $U = [(u_i)_i]$ where $(u_i)_i$ is the unconditional part of $(e_i)_i$. This is much easier and the isomorphism constant depends on the bound of the sequence $(m_j)_j$.
\end{rmk}

Given a Banach space with a 1-convex block homogeneous basis, by Proposition \ref{step3}, we may denote by $X_\infty$ the space $X_{(m_j)_j}$ for an arbitrary unbounded sequence of natural numbers. Up to a constant of 256 the choice of the sequence is not relevant.

\begin{prp}
\label{Xinftyusd}
Let $X$ be a Banach space with a 1-convex block homogeneous basis and let $(m_j)_j$ be an unbounded sequence of natural numbers. Then, the space $X_\infty = X_{(m_j)_j}$ admits an unconditional Schauder decomposition $(X_k)_k$ so that for each $k\in\N$ the spaces $X_\infty$ and $X_k$ are 256-isomorphic. Notationally,
\begin{equation}
\label{Xinftyisanaccordion}
X_\infty = \sum_{k=1}^\infty\oplus X_\infty\text{ unconditionally.} 
\end{equation}
\end{prp}

\begin{proof}
Partition the natural numbers into infinitely many infinite sets $(N_k)_k$ so that $(m_j)_{j\in N_k}$ is unbounded for all $k\inn$. By unconditionality we obtain that the spaces $(X_{(m_j)_{j\in N_k}})_{k\in\N}$ form an unconditional Schauder decomposition for $X_{(m_j)_{j\in\N}}$. By Proposition \ref{step3} the result follows.
\end{proof}

\begin{rmk}
\label{veryclosetoprimarynotquitethereyet}
If $X$ is a Banach space with a 1-convex block homogeneous basis $(e_i)_i$, by using Theorem \ref{there can be only one} and Remark \ref{if convex block homogeneous then all UFDDs are absorbed} we conclude that if $X = V\oplus W$ and $X$ is complemented in $V$ (it is always complemented in either $V$ or $W$) then there exists a complemented subspace $Y$ of $X_\infty$ so that $V\simeq X\oplus (X_\infty\oplus Y)$ (this decomposition requires the information about the kernel of certain projections given by Proposition \ref{ntwc has complemented subseq}). This is not quite enough to imply that $X$ is primary however it is very close to having this property. There is hope that a Pe\l cyz\'nski decomposition type argument \cite{P} can be used to show that $X_\infty\oplus Y\simeq X_\infty$. This would imply that $X$ is primary. The problem in this approach is the poor understanding of the ``outside'' norm in \eqref{Xinftyisanaccordion} (unless $X = J_p$, the jamesification of $\ell_p$, there is no outside norm in the strict sense).
\end{rmk}

\section{The Baire-1 functions of a space with a spreading basis}\label{Baire-1}

We denote by $\mathcal{B}_1(X)$ the subspace of $X^{**}$ that consists of all Baire-1 functions, i.e. those $x^{**}$ for which there is a sequence $(x_n)_n$ in $X$ with
$x^{**} = w^*$-$\lim_nx_n$. In this rather small section we include some observations concerning the position of $\mathcal{B}_1(X)$ in $X^{**}$ and of $X$ in $\mathcal{B}_1(X)$ whenever $X$ is a Banach space with a conditional spreading basis. We do not use any of these results in the rest of the paper, however, we think that they are of independent interest since they witness the highly canonical behavior exhibited by spaces with conditional spreading bases.

\begin{prp}\label{baire-1 are complemented in strongly summing}
Let $X$ be a Banach space with a strongly summing conditional spreading basis $(e_i)_i$ and denote $e^{**} = w^*$-$\lim_ie_i$. Then, the map $P:X^{**}\rightarrow X^{**}$ with
$$Px^{**} = w^*\text{-}\sum_ix^{**}(e_i^*)e_i + \left(\lim_ix^{**}\left(s_{(i,\infty)}\right)\right)e^{**}$$
is a bounded linear projection onto $\mathcal{B}_1(X)$.
\end{prp}

\begin{proof}
By Proposition \ref{equivalentfornotsumming} (v) and (iv) the limit $w^*$-$\sum_ix^{**}(e_i^*)e_i$ and the limit $\lim_ix^{**}(s_{(i,\infty)})$ exist. Hence, $P$ is well defined, bounded and maps into the space of Baire-1 functions of $X^{**}$.
It remains to show that $Px^{**} = x^{**}$ whenever $x^{**}$ is Baire-1, i.e. there is a sequence $(x_n)_n$ in $X$ with $x^{**} = w^*$-$\lim_nx_n$. For each $n\inn$ define $y_n = \sum_{i=1}^nx^{**}(e_i^*) + x^{**}(s_{(n,\infty)})e_{n+1}$.
Then
\begin{equation}
w^*\text{-}\lim_n y_n = Px^{**}\;\text{and}\;s(y_n) = x^{**}(s)\;\text{for all}\;n.
\end{equation}
If $(x_n - y_n)_n$ has a subsequence that converges to zero in norm, then $x^{**} = w^*$-$\lim_ny_n$ and there is nothing left to prove. Otherwise, the sequence $(x_n - y_n)_n$ is seminormalized and point-wise null, with respect to $(e_i)_i$,
and $s(x_n - y_n) = s(x_n) - x^{**}(s)\rightarrow 0$. A sliding hump argument yields that, passing to a subsequence, $(x_n - y_n)_n$ is equivalent to a block sequence $(w_n)_n$ with $s(w_n) = 0$ for all $n$.
By Proposition \ref{summingzero}, $(x_n - y_n)_n$ is unconditional, and since it is weak Cauchy, it is weakly null. In conclusion, $P^{**}x^{**} = w^*$-$\lim_ny_n = w^*$-$\lim_nx_n = x^{**}$.
\end{proof}

It was proved in \cite[Theorem 2.3 f), page 4]{FOSZ} that a Banach space with a conditional spreading basis not containing $c_0$ and $\ell_1$ is quasi-reflexive of order one. The following can be viewed as a generalization of this result.

\begin{cor}\label{codim-1 in baire-1 if no c0 subspace}
Let $X$ be a Banach space with a conditional spreading basis $(e_i)_i$. If $X$ contains no subspace isomorphic to $c_0$, then $X$ is of co-dimension one in  $\mathcal{B}_1(X)$.
\end{cor}

\begin{proof}
If $X$ does not contain $c_0$ then, by Proposition \ref{equivalentfornotsumming}, $(e_i)_i$ must be strongly summing. By Proposition \ref{baire-1 are complemented in strongly summing}, for each $x^{**}$ in $\mathcal{B}_1(X)$ we have that
$x^{**} = w^*$-$\sum_ix^{**}(e_i^*)e_i + (\lim_ix^{**}(s_{(i,\infty)}))e^{**}$. By Lemma \ref{xdoublestartrailsc0} we obtain that $\sum_ix^{**}(e_i^*)e_i$ is norm convergent for all $x^{**}$ in $X^{**}$, which implies that
$X = \mathcal{B}_1(X)\cap\mathrm{ker}(w$-$\lim_is_{(i,\infty)})$.
\end{proof}

\section{Spreading models of non-reflexive spaces}\label{nonreflexiveadmitcbhspreadingmodels}
In this section we show that every sequence generating a conditional spreading model has a block sequence of averages that generates a convex block homogeneous spreading model. We also observe that non-reflexive Banach spaces always have sequences that generate convex block homogeneous spreading models.

\begin{rmk}
\label{ntwcconvexblockdecreases}
Let $(x_i)_i$ be a non-trivial weak Cauchy sequence in some Banach space $X$ that generates a sequence $(e_i)_i$ as a spreading model. If $(y_i)_i$ is a convex block sequence of $(x_i)_i$ that generates a sequence $(\tilde{e}_i)_i$ as a spreading model, then the linear map $T: [(e_i)_i] \rightarrow [(\tilde{e}_i)_i]$ with $Te_i = \tilde e_i$ has norm at most one. Note that it is important for the sequence $(x_i)_i$ to already generate some spreading model.
\end{rmk}

\begin{lem}
Let $E_\xi$, $\xi<\omega_1$ be a transfinite hierarchy of Banach spaces so that each $E_\xi$ has a Schauder basis $(e_i^\xi)_i$. Assume moreover that for every $\xi<\zeta$ the linear map $T_{\xi,\zeta}: E_\xi \rightarrow E_\zeta$ defined by $T_{\xi,\zeta}e_i^\xi  = e_i^\zeta$ is well defined and has norm at most one. Then there exists $\xi_0<\omega_1$ such that for every $\xi_0\leqslant \xi<\omega_1$ the map $T_{\xi_0,\xi}$ is an isometry.\label{normstabilizes}
\end{lem}

\begin{proof}
If we assume that the conclusion does not hold, then by passing to an uncountable subset and relabeling we may assume that for every $\xi<\omega_1$ there exist $n_\xi\inn$  and rational numbers $(c_i^\xi)_{i=1}^{n_\xi}$ such that for every $\xi<\zeta<\omega_1$ the following holds.
\begin{equation}
\left\|\sum_{i=1}^{n_\xi}c_i^\xi e_i^\zeta\right\| < \left\|\sum_{i=1}^{n_\xi}c_i^\xi e_i^\xi\right\|.
\end{equation}
By passing to a further uncountable subset and relabeling once more we may assume that there exist $n\inn$ and rational numbers $(c_i)_{i=1}^{n}$
such that for every $\xi<\zeta<\omega_1$ the following holds.
\begin{equation}
\left\|\sum_{i=1}^{n}c_ie_i^\zeta\right\| < \left\|\sum_{i=1}^{n}c_ie_i^\xi\right\|.
\end{equation}
Setting $\alpha_\xi = \|\sum_{i=1}^{n}c_ie_i^\xi\|$ for $\xi<\omega_1$, we conclude that $(\alpha_\xi)_{\xi<\omega_1}$ is strictly decreasing which is absurd.
\end{proof}

\begin{prp}
\label{convexsmstabilizes}
Let $X$ be a Banach space and $(x_i)_i$ be a bounded sequence in $X$ without a weakly convergent subsequence. Then, there exists a convex block sequence $(\tilde x_i)_i$ of $(x_i)_i$ generating a spreading model $(e_i)_i$ so that the spreading model admitted by any further convex block sequence $(y_i)_i$ of $(\tilde x_i)_i$ is isometrically equivalent to $(e_i)_i$.
\end{prp}

\begin{proof}
By Rosenthal's $\ell_1$ theorem \cite{Rell1}, $(x_i)_i$ has either a subsequence that is equivalent to the unit vector basis of $\ell_1$ or it has  a non-trivial weak Cauchy subsequence. In either case, by passing to a subsequence, every further convex block sequence of $(x_i)_i$ has a subsequence generates a Schauder basic spreading model. Let us assume now that the conclusion is not satisfied. Using a transfinite recursion and  Remark \ref{ntwcconvexblockdecreases} we may construct a transfinite hierarchy   $(x_i^\xi)_i$ of convex block sequences of $(x_i)_i$, satisfying the following:
\begin{itemize}

\item[(i)] For every $\xi<\zeta<\omega_1$ the sequence $(x_i^\zeta)_i$ is eventually a convex block sequence of $(x_i^\xi)_i$ (``eventually'' is necessary to pass the properties to limit ordinals). 

\item[(ii)] For every $\xi<\omega_1$ the sequence $(x_i^\xi)_i$ generates a Schauder basic sequence $(e_i^\xi)_i$ as a spreading model.

\item[(iii)] For every $\xi<\zeta$ the natural linear map $T: [(e_i^\xi)_i] \rightarrow [(e_i^\zeta)_i]$ has norm at most one but it is not an isometry.

\end{itemize}
Lemma \ref{normstabilizes} and (iii) yield a contradiction.
\end{proof}

\begin{prp}
\label{avergestocbhsm}
Let $X$ be a Banach space and $(x_i)_i$ be a sequence in $X$ that generates a conditional spreading model $(e_i)_i$ and let also $(z_i)_i$ be the convex block homogeneous part of $(e_i)_i$. Then there exists a block sequence $(y_i)_i$ of averages of $(x_i)_i$ so that:
\begin{itemize}
 \item[(i)] the sequence $(y_i)_i$ generates a spreading model isometrically equivalent to $(z_i)_i$ and
 \item[(ii)] every convex block sequence $(w_i)_i$ of $(y_i)_i$ has a subsequence generates a spreading model isometrically equivalent to $(z_i)_i$.
\end{itemize}
\end{prp}

\begin{proof}
 By Rosenthal's $\ell_1$ theorem \cite{Rell1}, passing to a subsequence, the sequence $(x_i)_i$ is non-trivial weak Cauchy. Otherwise, if it had a weakly null subsequence, as it is well known, it would generates an unconditional spreading model. If it had a subsequence converging weakly to a non-zero element it would generate a singular spreading model or an $\ell_1$ spreading model (see \cite[Theorem 38, page 592]{AKT}), or if it had a norm convergent subsequence it would generate a trivial spreading model (i.e. a spreading sequence in seminormed space that is not a normed space). Find a convex block sequence $(\tilde x_i)_i$ of $(x_i)_i$ that satisfies the conclusion of Proposition \ref{convexsmstabilizes}. Denote by $(\tilde e_i)_i$ the spreading model generated by this sequence. The sequence $(x_i-\tilde x_i)_i$ is weakly null and, passing to a subsequence, it is either norm null or it generates some unconditional spreading model $(v_i)_i$. If it is norm null then the proof is complete. Otherwise, $(v_i)$ is not equivalent to the unit vector basis of $\ell_1$. If it were, then either $(e_i)_i$ or $(\tilde e_i)_i$ would have to be equivalent to the unit vector basis of $\ell_1$ and by Remark \ref{ntwcconvexblockdecreases} this cannot be the case. By Lemma \ref{rosenthaldichotomy} (applied to $(v_i)$) and a standard counting argument, passing to a subsequence of $(x_i - \tilde x_i)_i$ we have that for every $\e>0$ there exists $M_0,n_0\in\N$ so that for any $F\subset\N$ with $\min(F)\geqslant n_0$ and $\#F\geqslant M_0$ we have
\begin{equation}
\label{avergestocbhsm eq1}
\frac{1}{\#F}\left\|\sum_{i\in F}\left(x_i - \tilde x_i\right)\right\| < \e. 
\end{equation}
By Theorem \ref{spreading-characterization}, given $n\in\N$ and $\e>0$ there exists $M_n\in\N$ so that for all $a_1,\ldots,a_n\in[-1,1]$ and $M\geqslant M_n$ we have
\begin{equation}
\label{avergestocbhsm eq2}
\left|\left\|\sum_{i=1}^na_iz_i\right\| - \left\|\sum_{i=1}^na_i\left(\frac{1}{M}\sum_{j=((i-1)M+1)}^{iM}e_j\right)\right\|\right|<\e.
\end{equation}
Using \eqref{avergestocbhsm eq1}, the fact that the spreading model of $(\tilde x_i)_i$ is preserved when taking averages, and a limit argument one may deduce that for any $\e>0$ there exists $M_0\in\N$ so that for any $M\geqslant M_0$ and $a_1,\ldots,a_n\in[-1,1]$
\begin{equation}
\left|\left\|\sum_{i=1}^na_i\tilde e_i\right\| - \left\|\sum_{i=1}^na_i\left(\frac{1}{M}\sum_{j=((i-1)M+1)}^{iM}e_j\right)\right\|\right|<\e.
\end{equation}
It immediately follows that $(z_i)_i$ and $(\tilde e_i)_i$ are isometrically equivalent. Using \eqref{avergestocbhsm eq1} one may clearly now choose a sequence of averages of $(x_i)_i$ that satisfies the conclusion.
\end{proof}

\begin{cor}
\label{cbhsmofnr}
Let $X$ be a non-reflexive Banach space. Then there exists a sequence $(x_i)_i$ in $X$ that generates a 1-convex block homogeneous spreading model $(e_i)_i$. 
\end{cor}

\begin{proof}
For a bounded sequence without a weakly convergent subsequence apply Proposition \ref{convexsmstabilizes} to find a sequence $(x_i)_i$ generating a spreading model $(e_i)_i$ so that the conclusion of that proposition is satisfied. The argument used to obtain \eqref{avergestocbhsm eq2} yields that $(e_i)_i$ is 1-convex block homogeneous.
\end{proof}

\begin{rmk}
Note that it is not true that a non-reflexive Banach space $X$ must necessarily admit a conditional convex block homogeneous spreading model. It may very well be the case that $X$ only admits the unit vector basis of $\ell_1$ as a spreading model.
\end{rmk}

In \cite{AMusm} a Banach space $\mathfrak{X}_{\mathrm{usm}}$ is constructed that is hereditarily spreading model universal for all subsymmetric sequences. This means that for every possible subsymmetric sequence $(x_i)_i$ and every subspace $Y$ of $\mathfrak{X}_{\mathrm{usm}}$ there exists a sequence $(y_i)_i$ in $Y$ that generates a spreading model equivalent to $(x_i)_i$. This is a hereditarily indecomposable reflexive Banach space and it is constructed via a Tsirelson-type saturation method with constraints. We restate a problem posed in that paper that is relevant to this section.
\begin{prb}
Does there exist a Banach space $X$ that is hereditarily universal for all spreading sequences, i.e. both for conditional and unconditional ones?
\end{prb}
A space $X$ with the aforementioned property would have to be saturated with non-reflexive hereditarily indecomposable subspaces. Before closing this section it is worth mentioning that Banach spaces with a conditional spreading basis don't have a large variety of conditional spreading models. 
\begin{prp}
Let $X$ be a Banach space with a conditional spreading basis $(e_i)_i$ and let $(z_i)_i$ be its convex block homogeneous part. Let also $(x_i)_i$ be a sequence in $X$ that generates a conditional spreading model $(d_i)_i$. Then, the convex block homogeneous part of $(d_i)_i$ is either equivalent to $(z_i)_i$ or to the summing basis of $c_0$.
\end{prp}

\begin{proof}
Apply Proposition \ref{avergestocbhsm} to find a sequence $(y_i)_i$ that generates the convex block homogeneous part $(\tilde d_i)_i$ of $(d_i)_i$ and crucially also satisfies the second conclusion of that proposition. By Theorem \ref{maintheoremweakcauchy} (i) $(y_i)_i$ has a convex block sequence $(\tilde y_i)_i$ that is either equivalent to $(z_i)_i$ or to the summing basis of $c_0$. By Proposition \ref{avergestocbhsm} (ii) $(\tilde d_i)_i$ is equivalent to $(z_i)_i$ or equivalent to the summing basis of $c_0$.
\end{proof}

\begin{rmk}
The above proposition easily implies that there exists no space with a spreading Schauder basis that admits all conditional spreading models.
\end{rmk}

\section{Universality of $C(\omega^\omega)$ for all spreading models}\label{comegaomegaadmitsall}
It was proved by E. Odell in \cite[Proposition 5.10, page 419]{O} that every subsymmetric sequence is generated as a spreading model by some sequence in the space $C(\omega^\omega)$. In this section we prove that this can be extended to include conditional spreading sequences as well. This is interesting because the ordinal number $\omega^\omega$ is the first $\alpha$ for which $C(\alpha)$ is not isomorphic to $c_0(\N)$. Whereas the space $c_0(\N)$ admits only the unit vector basis of $c_0(\N)$ and the summing basis of $c_0(\N)$ as spreading models, the ``successor'' space $C(\omega^\omega)$ admits all possible spreading models. We mention here that in \cite[Proposition 63, page 606]{AKT} it was proved that $c_0(\N)$ admits all spreading sequences as 2-spreading models in the sense of \cite{AKT}. We first show the result for convex block homogeneous sequences. The final statement is then an easy consequence of Theorem \ref{spreading-characterization}. The proof requires a relatively short construction in the flavor of Schreier space \cite{S}. In this section we shall denote $e_i^*$ the elements of the unit vector basis of $\R^n$, of $c_{00}(\N)$, or of $\ell_\infty(\N)$.

We fix a Banach space $Z$ with a normalized bimonotone equal signs additive basis $(z_i)_i$ (recall that by Corollary \ref{renormto1cbh} any space with a convex block homogeneous basis may be renormed to have this property). Then, $(z_i)_i$ is 1-spreading and furthermore, by Remark \ref{esabimonotoneisthesameasbeingequaltoconditionaljamesification}, for all intervals $E_1<E_2<\cdots<E_n$ of $\N$ with $\max(E_k)+1 = \min(E_k)$, for $1\leqslant k<n$, and all scalars $a_1,\ldots,a_m$ we have
\begin{equation}
\label{self-jamesification}
\left\|\sum_{k=1}^n\left(\sum_{j\in E_k}a_j\right)z_k\right\| \leqslant \left\|\sum_{j=1}^ma_jz_j\right\|.
\end{equation}
For each $n\inn$ take a $(1/2n)$-dense finite set $F_n$ in the unit ball of the subspace spanned by the first $n$ vectors of $(z_i^*)_i$, the biorthogonal functionals to $(z_i)$. We require $F_n$ to satisfy the following.
\begin{itemize}
 \item[(a)] $F_n$ is symmetric and $z_i^*\in F_n$ for $1\leqslant i\leqslant n$,
 \item[(b)] $F_n$ is closed under projections onto intervals, and
 \item[(c)] if $x^* = \sum_{i=1}^na_iz_i^*$ is in $F_n$ and $1\leqslant k_0<n$, then $y^* = \sum_{i=1}^{k_0}a_iz^*_i + \sum_{i=k_0+1}^{n-1}a_{i+1}z_i^*$ is in $F_n$ as well.\label{propertycoffn}
\end{itemize}
Note that property (c) makes sense because, using only the fact that  $(z_i)_i$ is 1-spreading, if $x^*$, $y^*$ are as above then $\|y^*\|\leqslant \|x^*\|$. We furthermore require $F_n\subset F_{n+1}$ for all $n\inn$. 

In the definition bellow, by $\sum_{i\in A}e_i^*$ we mean the sequence of scalars $(c_i)_i$ in $\ell_\infty(\N)$ with $c_i = 1$ for $i\in A$ and $c_i = 0$ otherwise. The set $A$ may be finite or infinite.
Define the subset of the unit ball of $(\ell_\infty(\N),\|\cdot\|_{\ell_\infty})$
\begin{equation}
\label{omega to the omega (not so) fancily written}
\begin{split}
\mathcal{K}_Z = \Bigg\{& \sum_{k=1}^na_k\Bigg(\sum_{m_k\leqslant i < m_{k+1}}e_i^*\Bigg): n\leqslant m_1<\cdots<m_n < m_{n+1} \leqslant \infty,\\
& n\in\N\text{ and }\sum_{k=1}^na_kz_k^*\in F_n\Bigg\}.
\end{split}
\end{equation}
It is rather standard to verify that $\mathcal{K}_Z$ is a countable and compact subset of the unit ball of $\ell_\infty(\N)$ endowed with the topology of pointwise convergence and that the Cantor-Bendixson index of $\mathcal{K}_Z$ is $\omega+1$. In fact, the $\omega$'th derivative of $\mathcal{K}_Z$ is a the singleton that contains the zero element of $\ell_\infty(\N)$. To see this, note that an element of $\mathcal{K}_Z$ depends on a set $\{m_1<\cdots<m_n<m_{n+1}\}$ in $\N\cup\{\infty\}$ with $n\leqslant m_1$ and scalar coefficients $a_1,\ldots,a_n$ that may be chosen from a finite set that depends on $n$. Therefore, $\mathcal{K}_Z$ is homeomorphic to $\omega^\omega$ and the space $C(\mathcal{K}_Z)$ is isometric to $C(\omega^\omega)$.

We define a norm on $c_{00}(\N)$ by setting $\|x\|_{\tilde{\mathcal{S}}(Z)} = \sup\{|f(x)|: f\in\mathcal{K}_Z\}$, where for $f\in\mathcal{K}_Z$ and $x\in c_{00}(\N)$ the symbol $f(x)$ denotes the usual inner product of $f$ with $x$. We denote $\tilde{\mathcal{S}}(Z)$ to be the the completion of the space $(c_{00}(\N),\|\cdot\|_{\tilde{\mathcal{S}}(Z)})$. One may refer to the space $\tilde{\mathcal{S}}(Z)$ as the ``conditional schreierification of $Z$''. Note that the map $T:\tilde{\mathcal{S}}(Z)\to C(\mathcal{K}_Z)$ defined by $T(e_i)(f) = f(e_i)$ is a linear isometric embedding. We note for later use the following, which follows from $\mathcal{K}_Z^{(\omega)} = \{0\}$.

\begin{rmk}
\label{omegatotheomegamustgotozero}
Any onto homeomorphism $\phi:\omega^\omega\to \mathcal{K}_Z$ satisfies $\phi(\omega^\omega) = 0$. Hence, if $T:\tilde{S}(Z)\to C(\mathcal{K}_Z)$ is the above isometry and $\tilde T:\tilde{\mathcal{S}}(Z)\to C(\omega^\omega)$ is defined by $\tilde Tx(\alpha) = Tx(\phi(a))$, then for every $x\in\tilde{\mathcal{S}}(Z)$ we have $\tilde Tx(\omega^\omega) = 0$. That is,  $\tilde{\mathcal{S}}(Z)$ isometrically embeds into $C_0(\omega^\omega)$.
\end{rmk}

It remains to observe that the sequence $(e_i)_i$ endowed with $\|\cdot\|_{\tilde{\mathcal{S}}(Z)}$ admits $(z_i)_i$ as a spreading model. Indeed, \eqref{omega to the omega (not so) fancily written} easily implies that for $n < k_1 <\cdots <k_n$ and $c_1,\ldots,c_n\in\R$ setting $m_1=k_1,  \ldots, m_n=k_n$ we have
\begin{equation}
\label{lowerspreadingmodelestimatethisiseasy}
\left\|\sum_{i=1}^nc_ie_{k_i}\right\| \geqslant \max\left\{x^*\left(\sum_{i=1}^nc_iz_i\right): x^*\in F_n\right\} \geqslant \frac{2n-1}{2n}\left\|\sum_{i=1}^nc_iz_i\right\|,
\end{equation}
whereas on the other hand by \eqref{omega to the omega (not so) fancily written}, if $x = \sum_{i=1}^nc_kz_{k_i}$ we also obtain
\begin{equation}
\label{lowerspreadingmodelestimatethisiseasyaswell}
\begin{split}
\left\|\sum_{i=1}^nc_ie_{k_i}\right\| \leqslant \max\left\{\vphantom{\left\|\sum_{j=1}^ms_{E_j}(x)z_i\right\|}\right.&\left\|\sum_{j=1}^ms_{E_j}(x)z_i\right\|: (E_j)_{j=1}^m \text{ is a sequence}\\
  &\left.\vphantom{\left\|\sum_{j=1}^ms_{E_j}(x)z_i\right\|}\text{of consecutive intervals of }\N\right\}\\
 =& \left\|\sum_{i=1}^nc_iz_i\right\| \text{ (by \eqref{self-jamesification})}.
\end{split}
\end{equation}
The desired result follows from \eqref{lowerspreadingmodelestimatethisiseasy} and \eqref{lowerspreadingmodelestimatethisiseasyaswell}. We summarize what we have shown in the following statement.

\begin{rmk}
\label{S(X)isbimonotone}
Given a Banach space with a bimonotone equal signs additive basis $(z_i)_i$, the unit vector basis $(e_i)_i$ of $c_{00}(\N)$ endowed with $\|\cdot\|_{\tilde{\mathcal{S}}(Z)}$
\begin{itemize}
 \item[(i)] forms a bimonotone Schauder basis for the space $\tilde{\mathcal{S}}(Z)$,
 \item[(ii)] generates a spreading model isometrically equivalent to $(z_i)_i$, and
 \item[(iii)] the summing functional $s$ defined on it is bounded, in fact it has norm one.
\end{itemize}
To see (i), note that in \eqref{omega to the omega (not so) fancily written}, by property (b) of the sets $F_n$, the coefficients $(a_i)_{i=1}^n$ may be restricted to intervals of $\{1,\ldots,n\}$ by replacing the initial part and the tail part with zeros. This is possible because the sets $F_n$ are chosen to be closed under taking projections onto intervals.
\end{rmk}

Translating the above remark and using the isometric embedding $\tilde T$ of $\tilde{\mathcal{S}}(Z)$ into $C(\omega^\omega)$ we obtain the following.
\begin{prp}
\label{if its nice enough its isometric}
Let $Z$ be a Banach space with a bimonotone and equal signs additive Schauder basis $(z_i)_i$. Then, $C(\omega^\omega)$ contains a sequence $(f_i)_i$ that generates a spreading model isometrically equivalent to $(z_i)_i$.
\end{prp}


We now state and prove the main result of this section.

\begin{thm}
\label{comegaomegaadmitsallspreadingmodelsisomorphically}
Let $X$ be a Banach space with a spreading Schauder basis $(x_i)_i$. Then there exists a sequence $(y_i)_i$ in $C(\omega^\omega)$ that generates a spreading model equivalent to $(x_i)_i$.  
\end{thm}

\begin{proof}
As we mentioned earlier, the unconditional case was already proved by E. Odell. In the conditional case and assuming that $(x_i)_i$ is 1-spreading, apply Theorem \ref{spreading-characterization} and let $(u_i)_i$, $(z_i)_i$ be the unconditional and convex block homogeneous parts of $(x_i)_i$ respectively. Find sequences $(f_i)_i$ and $(g_i)_i$ in $C(\omega^\omega)$ generating $(u_i)_i$ and $(z_i)_i$ as spreading models. Then, the sequence $(f_i,g_i)_i$ in $C(\omega^\omega)\oplus C(\omega^\omega)$ generates a spreading model equivalent to $(x_i)_i$. Since $C(\omega^\omega)$ is isomorphic to its square the proof is complete.
\end{proof}

\begin{rmk}
A somewhat similar proof yields that for a countable ordinal number $\alpha$, the space $C(\omega^{\omega^\alpha})$ admits all Schauder basic spreading sequences as $\mathcal{S}_\alpha$-spreading models. That is, if $(x_i)_i$ is a Schauder basic spreading sequence then there exists a sequence $(f_i)_i$ in $C(\omega^{\omega^\alpha})$ and positive constants $\kappa$, $K$ so that for all $F\in\mathcal{S}_\alpha$ and choice of scalars $(a_i)_{i\in F}$ we have
$$\kappa\left\|\sum_{i\in F}a_i x_i\right\|\leqslant \left\|\sum_{i\in F}a_i f_i\right\|\leqslant K\left\|\sum_{i\in F}a_i x_i\right\|.$$
The proof requires a variation of the set $\mathcal{K}_Z$ with an $\mathcal{S}_\alpha$ condition. This has to do with the convex block homogeneous part of $(x_i)_i$ and a similar set has to be defined for the unconditional part as well.
\end{rmk}

We shall now orient our attention towards proving a slightly more precise statement of Theorem \ref{comegaomegaadmitsallspreadingmodelsisomorphically}. The reason for this is that we will require an extra condition for the sequence in $C(\omega^\omega)$ that generates the desired conditional spreading model, in order to prove the main result of Section \ref{quasireflexivesm}.

\begin{rmk}
\label{unconditional schreierification}
Given a Banach space $U$ with a suppression unconditional spreading basis $(u_i)_i$, one may define a norm $\|\cdot\|_{\mathcal{S}(U)}$ on $c_{00}(\N)$ with completion $\mathcal{S}(U)$ so that the unit vector basis $(e_i)_i$ of $c_{00}(\N)$, endowed with $\|\cdot\|_{\mathcal{S}(U)}$,
\begin{itemize}
 \item[(i)]  forms a suppression unconditional Schauder basis for the space $\mathcal{S}(U)$,
 \item[(ii)] it generates a spreading model isometrically equivalent to $(u_i)_i$, and
 \item[(iii)] it is isometrically equivalent to some sequence $(g_i)_i$ in $C_0(\omega^\omega)$.
\end{itemize}
The definition of the norm $\|\cdot\|_{\mathcal{S}(U)}$ is a simpler version of the definition of $\|\cdot\|_{\tilde{\mathcal{S}}(Z)}$ and it uses a set similar to $\mathcal{K}_Z$ that contains elements of the form $\sum_{i=1}^na_ie_{m_i}$ where $n\leqslant m_1<\cdots<m_n$ and $y^* = \sum_{k=1}^na_iu_i^*$ is in an appropriate subset $G_n$ of the unit ball of $U^*$ that is closed under projections onto subsets. This is a more precise statement than what was proved by Odell in \cite[Proposition 5.10, page 419]{O} and we shall use it in the sequel.
\end{rmk}

\begin{lem}
\label{conditionalschreierificationsuppressionunconditionalinkernelofsummingfunctional}
Let $Z$ be a Banach space with a bimonotone equal signs additive basis. Let $x = \sum_{i=1}^mc_ie_i$ be a vector in $c_{00}(\N)$ and $1\leqslant p\leqslant q\leqslant m$ be natural numbers so that $\sum_{i=p}^qc_i = 0$. Then, if $y = \sum_{1\leqslant i<p}c_ie_i + \sum_{q<i\leqslant m}c_ie_i$ we have $\|y\|_{\tilde{\mathcal{S}}(Z)}\leqslant \|x\|_{\tilde{\mathcal{S}}(Z)}$. In particular, any block sequence $(y_i)_i$ in $\tilde{\mathcal{S}}(Z)$, with $s(y_i) = 0$ for all $i\in\N$, is suppression unconditional.
\end{lem}

\begin{proof}
Let $f = \sum_{k=1}^na_k\sum_{m_k\leqslant j<m_{k+1}}e_j^*$ be in $\mathcal{K}_Z$, as in \eqref{omega to the omega (not so) fancily written}, i.e. $x^* = \sum_{k=1}^na_kz_k^*\in F_n$ and $n\leqslant m_1<\cdots<m_n<m_{n+1}\leqslant \infty$. We will show that there is $g\in\mathcal{K}_Z$ with $f(y) = g(x)$. We distinguish two cases. If there is $1\leqslant k\leqslant n$ so that $m_k\leqslant p$ and $q < m_{k+1}$ observe that $f(y) = f(x)$. Otherwise, set $k_1 = \min\{1\leqslant k\leqslant n:$ with $p\leqslant m_{k}<q\}$ and $k_2 = \max\{1\leqslant k\leqslant n:$ with $p\leqslant m_{k}\leqslant q\}$. We shall assume that $k_1 < k_2$ as the case $k_1 = k_2$ is treated slightly differently but very similarly.

By property (c) of $F_n$ (see page \pageref{propertycoffn}) we have that $y^* = \sum_{k=1}^{k_1}a_kz_k^* + \sum_{k=k_1+1}^{n-k_2+k_1+1}a_{k+k_2-k_1-1}z_k^*$ is in $F_n$. We define $(\tilde m_k)_{k=1}^{n + 1 - k_2 + k_1 + 1}$ as follows.
$$
\tilde m_k =
\left\{
	\begin{array}{ll}
		m_k  & \mbox{if } 1\leqslant k <k_1,\\
		p & \mbox{if } k=k_1,\\
		q & \mbox{if } k=k_1+1, \text{ and}\\
		m_{k+k_2-k_1-1} & \mbox{if } k_1+1<k\leqslant n +1 - k_2 + k_1.
	\end{array}
\right.
$$
Define
$$g = \sum_{k=1}^{k_1}a_k\sum_{\tilde m_k\leqslant j < \tilde m_{k+1}}e_j^* + \sum_{k=k_1+1}^{n-k_2+k_1+1}a_{k+k_2-k_1-1}\sum_{\tilde m_k\leqslant j < \tilde m_{k+1}}e_j^*.$$
Then, $g$ is in $\mathcal{K}_Z$. Some computations yields that $f(y) = g(x)$.
\end{proof}

In isomorphic terms, the only improvement of the following statement when compared to Theorem \ref{comegaomegaadmitsallspreadingmodelsisomorphically} is conclusion (iii). This is however a very important condition necessary to prove the main result of Section \ref{quasireflexivesm}.

\begin{prp}
\label{whatonecangetisometricallyincomegaomega}
Let $Z$ be a Banach space with a bimonotone equal signs additive basis $(z_i)_i$ and $U$ be a Banach space with a suppression unconditional spreading basis $(u_i)_i$. Denote by $(f_i)_i$ the basis of $\tilde{\mathcal{S}}(Z)$ and by $(g_i)_i$ the basis of $\mathcal{S}(U)$. Also denote by $(x_i)_i$ the sequence $(u_i,z_i)_i$ in $(U\oplus Z)_0$ and set $X = [(x_i)_i]$. Finally, by $(h_i)_i$ denote the sequence $(g_i,f_i)_i$ in $(\mathcal{S}(U)\oplus\tilde{\mathcal{S}}(Z))_0$ and set $\tilde X = [(h_i)_i]$. Then, the following hold.
\begin{itemize}
 \item[(i)] The sequence $(h_i)_i$ generates a spreading model isometrically equivalent to $(x_i)_i$.
 \item[(ii)] The summing functional defined on $(h_i)_i$ is bounded, in fact it has norm one.
 \item[(iii)] Every block sequence $(w_i)_i$ of $(h_i)_i$ with $s(w_i) = 0$ for all $i\in\N$ is suppression unconditional.
 \item[(iv)] The space $\tilde X$ embeds isometrically into $C(\omega^\omega)$. In particular, $\tilde X$ is $c_0$-saturated.
\end{itemize}
\end{prp}

\begin{proof}
The first statement is an immediate consequence of Remark \ref{S(X)isbimonotone} (ii) and Remark  \ref{unconditional schreierification} (ii). The second statement follows from Remark \ref{S(X)isbimonotone} (iii). The third statement follows from Remark \ref{unconditional schreierification} (i) and Lemma \ref{conditionalschreierificationsuppressionunconditionalinkernelofsummingfunctional}. To see the last statement recall that $(C_0(\omega^\omega)\oplus C_0(\omega^\omega))_0$ embeds isometrically into $C(\omega^\omega)$. By Remark \ref{omegatotheomegamustgotozero} the space $\tilde{\mathcal{S}}(Z)$ embeds isometrically into $C_0(\omega^\omega)$ and by Remark \ref{unconditional schreierification} (iii) $\mathcal{S}(U)$ embeds isometrically into $C_0(\omega^\omega)$ as well. We conclude that $(\mathcal{S}(U)\oplus\tilde{\mathcal{S}}(Z))_0$ embeds isometrically into $C(\omega^\omega)$ which yields the desired result.
\end{proof}

\section{Spreading models of quasi-reflexive spaces}\label{quasireflexivesm}
As it was proved in Section \ref{nonreflexiveadmitcbhspreadingmodels}, non-reflexive Banach spaces always admit convex block homogeneous spreading models whereas in Section \ref{comegaomegaadmitsall} we showed that $C(\omega^\omega)$ admits all possible spreading sequences as spreading models. In this section we show that every conditional spreading sequence is admitted as a spreading model of a Banach space that is quasi-reflexive of order one, i.e. a non-reflexive space that is almost reflexive.









The most difficult part of the present section is to prove the proposition below. It involves a Tsirelson-type construction with saturation under constraints. For now, we state it and use it to prove the main theorem of this section. We present the proof of the proposition later.

\begin{prp}
\label{saturationconstruction}
Let $X$ be a Banach space with a normalized bimonotone Schauder basis $(x_i)_i$ so that
\begin{itemize}

\item[(i)] the space $\ell_1$ does not embed into $X$ and

\item[(ii)] the summing functional $s:X\rightarrow\mathbb{R}$ with respect to the basis $(x_i)_i$ is bounded and any block sequence $(y_i)_i$ of $(x_i)_i$, with $s(y_i) = 0$ for all $i\in\N$, is suppression unconditional.

\end{itemize}
Then there exists a Banach space $\mathfrak{X}$ with a bimonotone Schauder basis $(e_i)_i$  satisfying the following.
\begin{itemize}

\item[(a)] The linear map $T:\mathfrak{X}\to X$ defined by $Te_i = x_i$ is bounded and has norm one.

\item[(b)] For any natural numbers $n <i_1<\cdots<i_n$ and real numbers $c_1,\ldots,c_n$ we have that $\|\sum_{k=1}^nc_kx_{i_k}\| = \|\sum_{k=1}^nc_ke_{i_k}\|$.

\item[(c)] The basis $(e_i)_i$ is boundedly complete.

\item[(d)] The summing functional $s:\mathfrak{X}\rightarrow\mathbb{R}$ with respect to the basis $(e_i)_i$ is bounded and every block sequence $(y_i)_i$ of $(e_i)_i$ with $s(y_i) = 0$ for all $i\in\N$ spans a reflexive subspace of $\mathfrak{X}$.

\end{itemize}
\end{prp}

We now state the main result of this section. This result has two statements, one for unconditional spreading sequences and one for conditional ones. We shall give a proof of the second statement, which is also the more difficult one and it uses Proposition \ref{saturationconstruction}. Afterwards, we provide a proof of Proposition \ref{saturationconstruction}. At the end of the section we shall point out some of the steps required to achieve the first statement of the following theorem.

\begin{thm}
\label{conditionalassmofquasireflexive}
Let $(x_i)_i$ be a spreading Schauder basic sequence.
\begin{itemize}

  \item[(i)] If $(x_i)_i$ is unconditional, then there exists a reflexive Banach space $X$ with an unconditional Schauder basis $(e_i)_i$ that generates a spreading model equivalent to $(x_i)_i$.

 \item[(ii)] If $(x_i)_i$ is conditional, then there exists a Banach space $X$ that is quasi-reflexive of order one with a Schauder basis $(e_i)_i$ that generates a spreading model equivalent to $(x_i)_i$.
 
\end{itemize}
\end{thm}

\begin{proof}[Proof of (ii)]
We may assume that $(x_i)_i$ is bimonotone 1-spreading. Using Theorem \ref{spreading-characterization} and Corollary \ref{renormto1cbh}, by passing to an equivalent norm on $X$, there are a suppression unconditional sequence $(u_i)_i$ and a bimonotone equal signs additive sequence $(z_i)_i$ spanning spaces $U$ and $Z$ respectively, so that $(x_i)_i$ is isometrically equivalent to the sequence $(u_i,z_i)_i$ in $(U\oplus Z)_0$. Take the space $\tilde X$ with basis $(h_i)_i$ given by Proposition \ref{whatonecangetisometricallyincomegaomega}. Apply Proposition \ref{saturationconstruction} to the space $\tilde X$ to obtain a space $\mathfrak{X}$ with a Schauder basis $(e_i)_i$ satisfying the conclusion of that Proposition. We claim that this space has the desired properties.

By Propositions \ref{whatonecangetisometricallyincomegaomega} (i) and \ref{saturationconstruction} (b) we deduce that $(e_i)_i$ generates $(x_i)_i$ as a spreading model. Furthermore, Propositions \ref{whatonecangetisometricallyincomegaomega} (ii) and \ref{saturationconstruction} (a) clearly yield that the summing functional is bounded on $(e_i)_i$ which implies that the sequence $(d_i)_i$ with $d_1 = e_1$ and $d_i = e_i - e_{i-1}$ for $i\geqslant 2$ is a Schauder basis for $\mathfrak{X}$. Indeed, if $(P_n)_n$ denotes the sequence of projections associated to $(e_i)_i$, then the sequence of projections $(Q_n)_n$ associated to $(d_i)_i$ is given by $Q_nx = P_nx + s_{[n+1,\infty)}(x)e_n$. By Proposition \ref{saturationconstruction} one may conclude that any skipped block sequence of $(d_i)_i$ spans a reflexive subspace of $\mathfrak{X}$. Since $\mathfrak{X}$ generates a conditional spreading model it cannot be reflexive and therefore by \cite[Theorem 2.1, page 3]{FOSZ} it is quasi-reflexive of order one.
\end{proof}

\subsection*{Proof of Proposition \ref{saturationconstruction}}
We shall brake up the proof of Proposition \ref{saturationconstruction} into several steps. Let us from now on fix a Banach space $X$ with a normalized bimonotone Schauder basis $(x_i)_i$ satisfying the assumptions of Proposition \ref{saturationconstruction}. Let us also fix strictly increasing sequences of natural numbers  $(m_j)_j, (n_j)_j$, with $m_1\geqslant 2, n_1\geqslant 4$, $m_j < n_j <m_{j+1}$ for all $j\inn$ and $\sum_j(m_j/n_j) < 1$. A norming set is a symmetric subset $G$ of the unit ball of $c_{00}(\N)$ that contains all elements of the unit vector basis $(e_i)_i$.

\begin{ntt}
If $G$ is a norming set, a functional
$$\al = \frac{1}{n}(f_1 + \cdots + f_k),$$
where $k\leqslant n\inn$ and $f_1 <\cdots <f_k\in G$, will be called an {\em $\al$-average} of $G$. The {\em size of $\al$} is defined to be $s(\al) = n$.

A sequence $\al_1 < \cdots < \al_k <\cdots$ of $\al$-averages will be called {\em very fast growing} if $s(\al_{k+1}) > \max\{ s(\al_k), 2^{\max\supp\al_k}\}$ for all $k$.

A functional
$$f = \frac{1}{m_j}\sum_{q=1}^d\al_q,$$
where $j\inn,\; d\leqslant n_j$ and $\al_1 < \cdots < \al_d$ is a very fast growing sequence of $\al$-averages of $G$, will be called a {\em weighted functional} of $G$. The {\em weight of $f$} is defined to be $w(f) = j$.
\end{ntt}

We shall recursively define an increasing sequence of subsets of $c_{00}(\N)$ and use their union to define a norm on $c_{00}(\N)$. Set
\begin{equation}
\label{W0conditonalcase}
W_0 = \left\{\sum_{i=1}^n\la_ie_i: \left\|\sum_{i=1}^n\la_ix_i^*\right\| \leqslant 1\right\}
\end{equation}
and assuming that $W_n$ has been defined, we set
\begin{eqnarray*}
W_{n+1}^\al &=& \{\al:\; \al\;\text{is an}\;\al\text{-average of}\;W_n\}\\
W_{n+1}^w &=& \{f:\; f\;\text{is a weighted functional of}\;W_n\}\\
W_{n+1} &=& W_n\cup W_{n+1}^\al\cup W_{n+1}^w
\end{eqnarray*}
Finally we set $W = \cup_{n=0}^\infty W_n$. For $x\in c_{00}(\N)$ we define
\begin{equation*}
\|x\| = \sup\left\{f(x): f\in W\right\}
\end{equation*}
and we set $\mathfrak{X}$ to be the completion of $(c_{00}(\N),\|\cdot\|)$.

\begin{rmks}
\label{basicremarks}
The following are easy observations that follow from the definition of $W$.
\begin{itemize}

\item[(i)] For every $f\in W$ and $E$ interval of the natural numbers, $f|_E$ is also in $W$, therefore the unit vector basis of $(e_i)_i$ forms a bimonotone Schauder basis for $\mathfrak{X}$.

\item[(ii)] Since $W_0\subset W$, the linear map $T:\mathfrak{X}\to X$ defined by $Te_i = x_i$ is bounded and it has norm one. It follows that the summing functional $s:\mathfrak{X}\to \R$ is bounded.

\item[(iii)] The set $W$ is the smallest norming set such that $W_0\subset W$ and for every $j\inn, d\leqslant n_j$ and  very fast growing sequence of $\al$-averages $\al_1<\cdots <\al_d$ of $W$, the functional $f = \frac{1}{m_j}\sum_{q=1}^d\al_q$ is also in $W$.

\end{itemize}
\end{rmks}

\begin{prp}
The unit vector basis of $(e_i)_i$ forms a boundedly complete Schauder basis for $\mathfrak{X}$.\label{boundedlycomplete}
\end{prp}

\begin{proof}
Towards a contradiction, let us assume that the conclusion fails. Then, there exists a block sequence $(y_k)_k$ and $\e >0$, such that the following hold.
\begin{itemize}

\item[(i)] For every $k\inn$\;$\|y_k\| > \e$, i.e. there exists $f_k\in W$ with $\ran f_k\subset \ran(y_k)$ and $f_k(y_k) > \e$.
    
\item[(ii)] For every $n\inn$ we have that $\|\sum_{k=1}^ny_k\|\leqslant 1$

\end{itemize}
Set $E_1 = \{1\}$ and choose a sequence $E_1<E_2<\cdots$ of intervals of the natural numbers satisfying the following.
\begin{itemize}

\item[(a)] For every $k\inn$, $\max E_k + 1 = \min E_{k+1}$.

\item[(b)] For every $k\inn$, $\#E_{k+1} > \max\{\#E_k, 2^{\max\supp y_{\max E_k}}\}$.

\end{itemize}
Set $w_k = \sum_{i\in E_k}y_i$ and $\al_k = (1/\#E_k)\sum_{i\in E_k}f_i$. The following hold.
\begin{itemize}

\item[($\alpha$)] For every $n\inn$, $\|\sum_{k=1}^nw_k\| = \|\sum_{i=1}^{\max E_n}y_i\| \leqslant 1$.

\item[($\beta$)] For every $k\inn$, $\al_k(w_k) > \e$.

\item[($\gamma$)] The sequence $(\al_k)_k$ is very fast growing.

\end{itemize}
Choose $j\inn$ such that $(\e n_j)/m_j > 1$. Then from ($\gamma$) we have that $f = \frac{1}{m_j}\sum_{k=1}^{n_j}\al_k$ is in $W$. From ($\beta$) we have that $f(\sum_{k=1}^{n_j}w_k) > 1$, i.e. $\|\sum_{k=1}^{n_j}w_k\| > 1$. This contradicts ($\alpha$) and the proof is complete.
\end{proof}

\begin{prp}
\label{equalonschreier}
For any natural numbers $n <i_1<\cdots<i_n$ and real numbers $c_1,\ldots,c_n$ we have that
$$\left\|\sum_{k=1}^nc_kx_{i_k}\right\| = \left\|\sum_{k=1}^nc_ke_{i_k}\right\|.$$
\end{prp}

\begin{proof}
Fix $n <i_1<\cdots<i_n$ and real numbers $c_1,\ldots,c_n$ such that if $y = \sum_{k=1}^nc_kx_{i_k}$, then $\|y\| = 1$. Remark \ref{basicremarks} (ii) yields that $\|\sum_{k=1}^nc_ke_{i_k}\| \geqslant 1$. We shall prove that $\|x\| \leqslant 1$, $x=\sum_{k=1}^nc_ke_{i_k}$, by inductively showing that for $f\in W_m$, $f(x) \leqslant 1$, for $m = 0,1,\ldots$.

For $f\in W_0$ this trivially follows by the fact that $\|\sum_{k=1}^nc_kx_{i_k}\| = 1$. Assume that the statement holds for every $f\in W_m$ and let $f\in W_{m+1}$. If $f$ is an $\al$-average of $W_m$, the result follows trivially from the inductive assumption. Otherwise, there exist $j\inn, d\leqslant n_j$ and $\al_1<\cdots<\al_d$ a very fast growing sequence of $\al$-averages of $W_m$ with $f = \frac{1}{m_j}\sum_{q=1}^d\al_q$.

Set $q_1 = \min\{q: \min\ran\al_q\geqslant n\}$. Then by the inductive assumption $\al_{q_1}(x) \leqslant 1$, while for $q<q_1, \al_q(x) = 0$. Moreover, by the very fast growing condition, for $q>q_1$ we have that $\|\al_q\|_\infty <1/2^{n+1}$ and therefore $\sum_{q>q_1}\al_q(x) < n/2^{n+1} < 1$. Combining the above, we conclude that $f(x) < (1/m_j)(1 + 1) \leqslant 1$.
\end{proof}

\begin{lem}
\label{averagesonaverages}
Let $\al$ be an $\alpha$-average in $W$ and $x_1<\cdots< x_k$ be block vectors in the unit ball of $\mathfrak{X}$ . Then $$\left|\al\left(\frac{1}{k}\sum_{i=1}^kx_i\right)\right| < \frac{1}{s(\al)} + \frac{2}{k}.$$
\end{lem}

\begin{proof}
Assume that $\al = (1/m)\sum_{j=1}^df_j$, where $f_1<\cdots<f_d\in W$ and $d\leqslant m = s(\al)$. Set
\begin{eqnarray*}
E_1 &=& \{i:\;\text{there exists at most one}\;j\;\text{such that}\;\ran(f_j)\cap\ran(x_i)\neq\varnothing\},\\
E_2 &=& \{1,\ldots,n\}\setminus E_1,\text{ and}\\
J_i &=& \{j: \ran(f_j)\cap\ran(x_i)\neq\varnothing\},\text{ for } i=1,\ldots,k.
\end{eqnarray*}
Then it is easy to see that
$$\left|\al\left(\frac{1}{k}\sum_{i\in E_1}x_i\right)\right|\leqslant \frac{1}{m}$$
and
$$\left|\al\left(\frac{1}{k}\sum_{i\in E_2}x_i\right)\right| \leqslant \frac{1}{m}\sum_{i\in E_2}\frac{1}{k}\left(\sum_{j\in J_i}\left|f_j(x_i)\right|\right) < \frac{2m}{m}\frac{1}{k}.$$
\end{proof}

\begin{prp}
\label{noell1}
The space $\ell_1(\N)$ does not embed into $\mathfrak{X}$.
\end{prp}

\begin{proof}
Towards a contradiction, assume that there exists a normalized block sequence $(z_k)_k$ in $\mathfrak{X}$, equivalent to the unit vector basis of $\ell_1(\N)$. Define $Y = \overline{\langle z_k: k\inn\rangle}$ and take $T|_Y:Y\to X$ to be the map from Remark \ref{basicremarks} (ii), which is bounded (recall, $Te_i = x_i$, so essentially $T$ is the identity map). By assumption $\ell_1(\N)$ does not embed into $X$ and hence $T$ is strictly singular. We may therefore choose a further normalized block sequence $(y_k)_k$ of $(z_k)_k$ with $\sum_k\|Ty_k\| < 1$.

Choose increasing subsets of the natural numbers $F_1< F_2 <\cdots$ satisfying the following. If $j_k = \max\supp(y_{\max F_k})$, then $\#F_{k+1} = 2k n_{j_k}$. Set $w_k = (1/\#F_k)\sum_{i\in F_k}y_i$. Then $(w_k)_k$ is seminormalized, equivalent to the unit vector basis of $\ell_1$ and it satisfies the following.
\begin{itemize}

\item[(i)] $\sum_k\|Tw_k\| < 1$ and hence for any scalar sequence $(\la_k)_k$ in $[-1,1]$ we have $\|\sum_k\la_kTw_k\| < 1$.
    
\item[(ii)] For every $\al$-average $\al$ in $W$ and $k\geqslant 2$ we have that
$$|\al(w_k)| < \frac{1}{s(\al)} + \frac{1}{k n_{j_{k-1}}}.$$

\end{itemize}
The second statement follows from Lemma \ref{averagesonaverages}. Consider now natural numbers $m \leqslant k_1 <\cdots <k_m$ and $\la_1,\ldots,\la_m$ in $[-1,1]$. We shall inductively prove the following. If $x = \sum_{i=1}^m\la_iw_{k_i}$, then $|f(x)| < 8$ for every $f\in W_n, n=0,1,\ldots$. This clearly contradicts the assumption that $(w_k)_k$ is equivalent to the unit vector basis of $\ell_1$.

For $f\in W_0$ this clearly follows from (i) and the obvious fact that $\|Tx\| = \sup\{|f(x)|: f\in W_0\}$. Assume that it holds for every $f\in W_m$ and let $f\in W_{m+1}$. If $f$ is an $\al$-average then there is nothing to prove. Otherwise, there are $t\inn, d\leqslant n_t$ and a very fast growing sequence of $\al$-averages $\al_1<\cdots<\al_d$ of $W_m$ with $f = (1/m_t)\sum_{q=1}^d\al_q$. Set $i_0 = \min\{i\leqslant m: j_{k_i} > t\}$. By the choice the sequence $(j_k)_k$, we have that
$$\|f\|_\infty \leqslant \frac{1}{m_t}\leqslant \frac{1}{t} \leqslant \frac{1}{j_{k_{i_0-1}}} = (\max\supp w_{k_{i_0-1}})^{-1}$$
and hence we conclude the following.
\begin{equation}
\label{noell1eq1}
\left|f\left(\sum_{i < i_0}\la_iw_{k_i}\right)\right|\leqslant \|f\|_\infty\max\supp w_{k_{i_0-1}} \leqslant 1.
\end{equation}
Set $q_1 = \min\{q\leqslant d: \max\supp\al_q\geqslant\max\supp w_{k_{i_0+1}}\}$. Then, for $q>q_1$, we have that $$s(\al_q) > 2^{\max\supp w_{k_{i_0+1}}} \geqslant 2^{\#F_{k_{i_0+1}}} = 2^{2k_0\cdot n_{j_{k_{i_0}}}} \geqslant 2k_02^{n_{j_{k_{i_0}}}} > m2^{n_t} > mn_t.$$
Using (ii) and the above, for $q>q_1$ and $i>i_0 + 1$ we have that
$$\left|\al_q\left(w_{k_i}\right)\right| < \frac{1}{mn_t} + \frac{1}{m\cdot n_{j_{k_{i_0}}}} < \frac{2}{mn_t}.$$
We conclude the
\begin{equation}
\label{noell1eq2}
\left|\frac{1}{m_t}\sum_{q>q_1}\al_q\left(\sum_{i>i_0+1}\la_qw_{k_i}\right)\right| < \frac{1}{2}mn_t\frac{2}{mn_t} = 1.
\end{equation}
The inductive assumption yields that $|\al_{q_1}(\sum_{i>i_{0}+1}\la_qw_{k_i})| < 8$. Combining this with \eqref{noell1eq2} we obtain
\begin{equation}
\label{noell1eq3}
\left|f\left(\sum_{i>i_0+1}\la_qw_{k_i}\right)\right| = \left|\frac{1}{m_t}\sum_{q\geqslant q_1}\al_q\left(\sum_{i>i_0+1}\la_qw_{k_i}\right)\right| \leqslant 5.
\end{equation}
Finally, using the fact that $|f(\la_{i_0}w_{k_{i_0}} + \la_{i_0+1}w_{k_{i_0+1}})| \leqslant 2$, \eqref{noell1eq1} and \eqref{noell1eq3} we conclude that $|f(x)| < 8$.
\end{proof}

\begin{prp}
\label{reflexive}
Every block sequence $(z_k)_k$ of $(e_i)_i$ in $\mathfrak{X}$ with  $s(z_k) = 0$ for all $k\in\N$ spans a reflexive subspace of $\mathfrak{X}$.
\end{prp}

\begin{proof}
Let $(z_k)_k$ be a normalized block sequence in $\ker s$. By Proposition \ref{boundedlycomplete}, $(z_k)_k$ is boundedly complete. It remains to prove that it is shrinking as well. Towards a contradiction, assume that this is not the case, i.e. there exist a linear functional $f:\mathfrak{X}\to\R$ with $\|f\| = 1$, a further normalized block sequence  $(y_k)_k$ of $(z_k)_k$, and $\e>0$ with $f(y_k) > \e$ for all $k\inn$.

Choose a summable sequence of positive reals $(\e_k)_k$ with $\sum_{j>k}\e_j < \e_k$ for all $k\inn$. Moreover, choose increasing subsets $F_1<F_2<\cdots$ of the natural numbers with $\#F_k > 2\e_k^{-1}$ for all $k\inn$ and set $w_k = (1/\#F_k)\sum_{i\in F_k}y_i$. Then $\e < f(w_k)\leqslant\|w_k\|\leqslant 1$ for all $k\inn$. We shall prove that $(w_k)_k$ is unconditional and therefore equivalent to the unit vector basis of $\ell_1$. This contradicts Proposition \ref{noell1} and the proof will be finished.

We shall prove by induction on $n$ that for every $f\in W_n$, $\ell\leqslant m\inn$, $G\subset \{\ell,\ldots,m\}$ and $\la_\ell,\ldots,\la_m\subset[-1,1]$, there exists $g\in W_n$ satisfying the following.
\begin{itemize}

\item[(i)] $\ran(g)\subset \ran(f)$.

\item[(ii)] If $f$ is an $\al$-average, then $g$ is also an $\al$-average of size $s(g) = s(f)$
    
\item[(iii)] $f(\sum_{k\in G}\la_kw_k) < g(\sum_{k=\ell}^m\la_kw_k) + \e_\ell$

\end{itemize}
The third assertion of the above statement implies that $(w_k)_k$ is unconditional with suppression constant at most $1+\e_1$ which is the desired result.

We proceed to the inductive proof. Let $f\in W_0$, $\ell,m\in\N$ with $\ell\leqslant m$, $G\subset \{\ell,\ldots,m\}$ and $\la_\ell,\ldots,\la_m\subset[-1,1]$. Set
\begin{eqnarray*}
k_1 &=& \min\{k\in G: \min\supp f\leqslant \max\supp w_k\}\\
k_2 &=& \max\{k\in G: \max\supp f\geqslant\min\supp w_k\}\\
i_1 &=& \min\{i\in F_{k_1}: \min\supp f\leqslant \max\supp y_i\}\\
i_2 &=& \max\{i\in F_{k_2}: \max\supp f\geqslant\min\supp y_i\} 
\end{eqnarray*} 
For $k_1<k<k_2$ set $\tilde{w}_k = w_k$. Set $\tilde{G} = \{k\in G: k_1\leqslant k\leqslant k_2\}$,  and also set
\begin{eqnarray*}
 \tilde{w}_{k_1} &=& \frac{1}{\#F_{k_1}}\sum_{\{i\in F_{k_1}: i>i_1\}}y_i\\
 \tilde{w}_{k_2} &=& \frac{1}{\#F_{k_2}}\sum_{\{i\in F_{k_2}: i<i_2\}}y_i
\end{eqnarray*}
Then,
\begin{equation*}
\begin{split}
f\left(\sum_{k\in G}\la_kw_k\right) &= f\left(\sum_{k\in \tilde{G}}\la_kw_k\right)\\
&= f\left(\sum_{k\in \tilde{G}}\la_k\tilde{w}_k\right) + f\left(\la_{k_1}\frac{1}{\#F_{k_1}}y_{i_1}\right) + f\left(\la_{k_2}\frac{1}{\#F_{k_2}}y_{i_2}\right).
\end{split}
\end{equation*}
Using the fact that for $k\geqslant\ell$, $\#F_k > 2\e_\ell^{-1}$ we conclude
\begin{equation}
f\left(\sum_{k\in G}\la_kw_k\right) < f\left(\sum_{k\in \tilde{G}}\la_k\tilde{w}_k\right) + \e_{\ell}\label{reflexiveeq1}
\end{equation}
Recall that every block sequence of $(x_i)_i$ in $X$ that is in $\ker s$ is suppression unconditional. Therefore
$$f\left(\sum_{k\in \tilde{G}}\la_k\tilde{w}_k\right) \leqslant \left\|\sum_{k\in \tilde{G}}\la_kT\tilde{w}_k\right\| \leqslant \left\|\sum_{k=k_1}^{k_2}\la_kT\tilde{w}_k\right\|.$$
Choose $g^\prime\in W_0$ with $g^\prime(\sum_{k=k_1}^m\la_k\tilde{w}_k) = \|\sum_{k=k_1}^{k_2}\la_kT\tilde{w}_k\|$. Then, using \eqref{reflexiveeq1} we have that
\begin{equation}
f\left(\sum_{k\in G}\la_kw_k\right) < g^\prime\left(\sum_{k=k_1}^{k_2}\la_k\tilde{w}_k\right) + \e_{\ell}\label{reflexiveeq2}
\end{equation}
Set $g$ to be $g^\prime|_E$, where
$$E = \{\max\supp (y_{i_1})+1,\ldots,\min\supp (y_{i_2})-1\}\subset\ran(f).$$
The fact that $(x_i)_i$ is a bimonotone basis for $X$, yields that $g$ is in $W_0$. Moreover, $g(\sum_{k=\ell}^m\la_kw_k) = g^\prime(\sum_{k=k_1}^{k_2}\la_k\tilde{w}_k)$. Using \eqref{reflexiveeq2} we deduce that $g$ is the desired functional.

Assume now that the statement holds for $m\inn$ and let $f\in W_{m+1}$, $\ell\leqslant m\inn$, $G\subset \{\ell,\ldots,m\}$ and $\la_\ell,\ldots,\la_m\subset[-1,1]$. Consider first the case in which $f$ is an $\al$-average of $W_m$, i.e. there are $d$, $p\in\N$ with $d\leqslant p$ and $f_1<\cdots<f_d\in W_m$ with $f = (1/p)(f_1+\cdots+f_d)$. By the inductive assumption, for $q=1,\ldots,d$ there exist $g_i\in W_n$ satisfying the following.
\begin{itemize}

\item[(a)] $\ran (g_i)\subset \ran (f_i)$.

\item[(b)] $f_i(\sum_{k\in G}\la_kw_k) < g_i(\sum_{k=\ell}^m\la_kw_k) + \e_\ell$

\end{itemize}
We conclude that $g = \frac{1}{p}(g_1+\cdots+g_d)$ is the desired functional.

Consider now that case in which $f$ is a weighted functional in $W_m$, i.e. there are $j\inn$, $d\leqslant n_j$, and a very fast growing sequence $\al_1<\cdots<\al_d$ of $\al$-averages of $W_m$ with $f = (1/m_j)(\al_1+\cdots +\al_d)$. Set
\begin{equation*}
H = \{q\leqslant d: \ran(\al_q)\cap\ran (w_k) = \varnothing\;\text{for some}\;k\in G\}
\end{equation*}
If $H = \{q_1 < \cdots < q_p\}$, then clearly $(\al_{q_i})_{i=1}^p$ is very fast growing and also
\begin{equation}
f\left(\sum_{k\in G}\la_kw_k\right) = \frac{1}{m_j}\sum_{i=1}^p\al_{q_i}\left(\sum_{k\in G}\la_kw_k\right)\label{reflexiveeq3}
\end{equation}
Partition $H$ into two sets as follows.
\begin{eqnarray*}
H_1 &=& \{q_i\in H: \ran(\al_{q_i})\cap\ran (w_k) = \varnothing\;\text{for all}\;k\in\{\ell,\ldots,m\}\setminus G\}\\
H_2 &=& H\setminus H_1 
\end{eqnarray*}
For $q_i\in H_2$, set $\ell_i = \min\{k\in\{\ell,\ldots,m\}: \ran(\al_{q_i})\cap\ran (w_k)\neq\varnothing\}$. Notice that for $q_i < q_j\in H_2$ we have $\ell_1 < \ell_2$. For $q_i\in H_2$, set $G_i = G\cap\{\ell_i,\ldots,m\}$. Then $\al_{q_i}(\sum_{k\in G}\la_kw_k) = \al_{q_i}(\sum_{k\in G_i}\la_kw_k)$ and by the previous case, there exists an $\al$-average $\tilde{\al}_{q_i}$ of $W_n$ with $\ran(\tilde{\al}_{q_i}) \subset \ran (\al_{q_i})$, $s(\tilde{\al}_{q_i}) = s(\al_{q_i})$ and $\al_{q_i}(\sum_{k\in G_i}\la_kw_k) < \tilde{\al}_{q_i}(\sum_{k = \ell_i}^m\la_kw_k) + \e_{\ell_i}$. Clearly, $\tilde{\al}_{q_i}(\sum_{k = \ell_i}^m\la_kw_k) = \tilde{\al}_{q_i}(\sum_{k = \ell}^m\la_kw_k)$. We have concluded that for $q_i\in H_2$, $\tilde{\al}_{q_i}$ is an $\al$-average satisfying the following statements.
\begin{itemize}

\item[($\alpha$)] $\ran(\tilde{\al}_{q_i}) \subset \ran (\al_{q_i})$,

\item[($\beta$)] $s(\tilde{\al}_{q_i}) = s(\al_{q_i})$, and

\item[($\gamma$)] $\al_{q_i}(\sum_{k\in G}\la_kw_k) < \tilde{\al}_{q_i}(\sum_{k = \ell}^m\la_kw_k) + \e_{\ell_i}$.

\end{itemize}
For $q_i\in H_1$ set $\tilde{\al}_{q_i} = \al_{q_i}$. Then $(\tilde{\al}_{q_i})_{i=1}^p$ is very fast growing and hence $g = (1/m_j)\sum_{i=1}^p\tilde{\al}_{q_i}$ is in $W_{n+1}$ and $\ran g\subset\ran f$. Combining \eqref{reflexiveeq3}, the fact that for $q_i < q_j\in H_2$ we have that $\ell_1 < \ell_2$, the fact that $\sum_{k=\ell}^m\e_k < 2\e_\ell$ and ($\gamma$), we conclude the following.
\begin{eqnarray*}
f\left(\sum_{k\in G}\la_kw_k\right) &=& \frac{1}{m_j}\sum_{i=1}^p\al_{q_i}\left(\sum_{k\in G}\la_kw_k\right)\\
 &<& \frac{1}{m_j}\sum_{i=1}^p\tilde{\al}_{q_i}\left(\sum_{k = \ell}^m\la_kw_k\right) + \frac{1}{m_j}\sum_{k=\ell}^m\e_k\\
 &=& g\left(\sum_{k = \ell}^m\la_kw_k\right) + \frac{1}{m_j}\sum_{k=\ell}^m\e_k\\
 &<& g\left(\sum_{k = \ell}^m\la_kw_k\right) + \e_\ell
\end{eqnarray*}
The inductive step is now complete and so is the proof.
\end{proof}

To the best of our knowledge, it was not known whether every subsymmetric sequence is generated as a spreading model by an unconditional basis of a reflexive space. Recall that in \cite{AMusm} a reflexive Banach space is constructed that admits all subsymmetric sequences as spreading models in all of its subspaces. However, this space is hereditarily indecomposable, i.e. it does not contain unconditional basic sequences and its construction is rather complicated. We give a short description of  Theorem \ref{conditionalassmofquasireflexive} (i).

\begin{proof}[Proof of \ref{conditionalassmofquasireflexive} (i)]
We may assume that $(x_i)_i$ 1-subsymmetric and we take the space $\mathcal{S}(U)$ from Remark \ref{unconditional schreierification} with an unconditional Schauder basis $(e_i)_i$ generating a spreading model isometrically equivalent to $(x_i)_i$. Then, we define $W_0$ similarly as in \eqref{W0conditonalcase}, using instead the unit ball of $\mathcal{S}(U)$. The sets $W_m$ are then defined in the same way and their union gives a norming set $W$ that then defines the norm of the space $\mathfrak{X}$ with an unconditional basis that generates $(x_i)_i$ as a spreading model. This is proved identically as Proposition \ref{equalonschreier}. Proofs identical to those of Propositions \ref{boundedlycomplete} and \ref{noell1} yield that $\mathfrak{X}$ does not contain $c_0$ and $\ell_1$ and hence it is reflexive.
\end{proof}

\section{The diversity of convex block homogeneous bases}\label{counterexamplesection}
In this final section we attempt to give an answer to the question of what different types of convex block homogeneous bases there are. As it was observed in section \ref{convex block homogeneous bases}, one way of defining a convex block homogeneous basis is to take the jamesification of a subsymmetric basis. One  may ask whether this is the unique way of obtaining such bases. By Theorem \ref{spreading-characterization}, a positive answer to this question would imply that every conditional spreading sequence is defined via two subsymmetric sequences alone. As it turns out however, this is false as there exist convex block homogeneous bases that cannot be obtained by jamesifying subsymmetric sequences. We provide a fairly simple example  obtained by duality.


\begin{rmks}
\label{remarksjamesification}
Let $X,$ be a Banach space with a 1-subsymmetric (i.e. 1-unconditional and 1-spreading) basis $(x_i)_i$. Let $J(X)$ denote the jamesification of $X$ and $(e_i)_i$ its natural Schauder basis. The following can be shown easily. 
\begin{itemize}

\item[(i)] If $(x_i)_i$ is not equivalent to unit vector basis of $\ell_1(\N)$, then $(e_i)_i$ is conditional and spreading.

\item[(ii)] The sequence $(e_i)_i$ is 1-convex block homogeneous and bimonotone.

\item[(iii)] The unconditional part $(u_i)_i$ of $(e_i)_i$, i.e. the sequence $(e_{2i} - e_{2i-1})_i$, is equivalent to the sequence $(x_i)_i$. In particular, for any real numbers $a_1,\ldots,a_n$ we have
    \begin{equation}
    \label{differences2equivalent}
    \left\|\sum_{i=1}^na_ix_i\right\| \leqslant \left\|\sum_{i=1}^na_iu_i\right\| \leqslant 2 \left\|\sum_{i=1}^na_ix_i\right\|.
    \end{equation}
\end{itemize}
\end{rmks}

\begin{rmk}
\label{youcanonlybethejamesificationofyourunconditionalpart}
Remark \ref{remarksjamesification} (iii) implies the following. If $Z$ is a Banach space with a convex block homogeneous basis $(z_i)_i$ that is equivalent to the jamesification of some subsymmetric sequence $(x_i)_i$, then actually $(x_i)_i$ is equivalent to the unconditional part $(u_i)_i$ of $(z_i)_i$. In other words, if $(z_i)_i$ is equivalent to the jamesification of some subsymmetric sequence, then it is actually equivalent to the jamesification of its own unconditional part.
\end{rmk}

\subsection*{A convex block homogeneous basis by duality}
Let us denote by $(e_i)_i$ the boundedly complete basis of James space $J$. The term ``jamesification'' actually stems from this basis and it is the jamesification of the unit vector basis of $\ell_2(\N)$, i.e. for a sequence of scalars $(a_i)_i$ one has
\begin{equation}
\label{jamesboundedlycompletenorm}
\left\|\sum_ia_ie_i\right\| = \sup\left\{\left(\sum_{k=1}^n\left(\sum_{i\in E_k}a_i\right)^2\right)^{1/2}\right\},
\end{equation}
where the supremum is taken over all possible choices of successive intervals $(E_k)_{k=1}^n$ of $\N$ that are allowed to have gaps between them (although this makes no difference). The basis $(e_i)_i$ of James space is 1-convex block homogeneous and so by Proposition \ref{dualizecbh} the sequence $(s_{[1,n]})_n$, with $s_{[1,n]} = \sum_{i=1}^ne_i^*$ for all $n\in\N$, is convex block homogeneous. Let us denote by $Y$ the closed linear span of the sequence $(s_{[1,n]})_n$. This is a subspace of co-dimension one in $J^*$ and in fact it is isomorphic to $J^*$. As it was proved by R. C. James in \cite{J2}, the spaces $J$ and $J^*$ are not isomorphic. This in particular means that $(s_{[1,n]})_n$ is not equivalent to $(e_i)_i$, i.e., $(s_{[1,n]})_n$ is not equivalent to the jamesification of the unit vector basis of $\ell_2(\N)$.

\begin{prp}
\label{dual of James is not jamesification}
The basis $(s_{[1,n]})_n$ of $Y$ is convex block homogeneous but not equivalent to the jamesification of any subsymmetric sequence. 
\end{prp}

\begin{proof}
Assume that there exists a subsymmetric sequence $(x_i)_i$ so that $(s_{[1,n]})_n$ is equivalent to the jamesification of $(x_i)_i$. By Remark \ref{youcanonlybethejamesificationofyourunconditionalpart}, $(x_i)_i$ is equivalent to the sequence $(s_{[1,2n]} - s_{[1,2n-1]})_n$, i.e. the sequence $(e_{2n}^*)_n$. We will show that $(e_{2n}^*)_n$ is equivalent to the unit vector basis of $\ell_2(\N)$. This contradicts the last sentence of the preceding discussion.

Observe that by \eqref{jamesboundedlycompletenorm} the basis $(e_i)_i$ of $J$ dominates the unit vector basis of $\ell_2(\N)$ with constant one. Hence, $(e_n^*)_n$ and in extend also $(e_{2n}^*)_n$ is 1-dominated by the unit vector basis of $\ell_2(\N)$. For the inverse domination, let $a_1,\ldots,a_n$ be scalars the squares of which sum up to one, we will evaluate the norm of $x^* = \sum_{i=1}^na_ie_{2i}^*$. Consider the vector $x = \sum_{i=1}^na_i(e_{2i} - e_{2i-2})$ in $J$. Then, by \eqref{jamesboundedlycompletenorm} it follows that $1\leqslant\|x\|\leqslant \sqrt{2}$ and an easy calculation yields $x^*(x) = 1$. In conclusion, $1/\sqrt{2}\leqslant \|x^*\|\leqslant 1$. Therefore, $(e_{2n}^*)_n$ is equivalent to the unit vector basis of $\ell_2(\N)$ which is absurd.
\end{proof}

\end{document}